\documentclass[11pt]{article} 
\usepackage[utf8]{inputenc} 
\usepackage[margin=1in]{geometry} 
\usepackage{graphicx} 

\usepackage{booktabs} 
\usepackage{array} 
\usepackage{paralist} 
\usepackage{verbatim} 
\usepackage{mathrsfs}
\usepackage{amssymb}
\usepackage{amsthm}
\usepackage{amsmath,amsfonts,amssymb}
\usepackage{esint}
\usepackage{graphics}
\usepackage{enumerate}
\usepackage{mathtools}
\usepackage{xfrac}
\usepackage{nicefrac}
\usepackage{subcaption}
\usepackage[normalem]{ulem}
\usepackage{cancel}
\usepackage{enumitem}
\usepackage{fancyvrb}

\usepackage[backend=bibtex, style=alphabetic, giveninits=true, maxbibnames = 10, minbibnames = 10, maxalphanames = 3, minalphanames = 3, url = false]{biblatex}
\DeclareFieldFormat[article, unpublished]{title}{#1}

\let\olddiamond\diamond
\let\oldsquare\square 

\usepackage{mathabx}

\renewcommand{\square}{\oldsquare}
\renewcommand{\diamond}{\olddiamond}

\usepackage[usenames,dvipsnames]{xcolor}
\usepackage[colorlinks=true, pdfstartview=FitV, linkcolor=blue, citecolor=blue, urlcolor=blue]{hyperref}
\usepackage[normalem]{ulem}

\usepackage{tikz}
\usetikzlibrary{calc}
\usepackage{pgf}
\usetikzlibrary{external}
\numberwithin{equation}{section}
\numberwithin{figure}{section}

\newtheorem{theorem}{Theorem}[section]

\newtheorem{corollary}[theorem]{Corollary}
\newtheorem{proposition}[theorem]{Proposition}
\newtheorem{lemma}[theorem]{Lemma}

\theoremstyle{definition}

\newtheorem{example}[theorem]{Example}

\theoremstyle{theorem}

\theoremstyle{plain}

\newcommand*{\Id}{\ensuremath{\mathrm{I}_d}}

\newcommand*{\Zd}{\ensuremath{\mathbb{Z}^d}}

\newcommand{\eps}{\varepsilon}

\renewcommand{\P}{\ensuremath{\mathbb{P}}}

\renewcommand{\b}{\ensuremath{\mathbf{b}}}

\newcommand{\s}{\mathbf{s}}

\newcommand{\ep}{\eps}

\renewcommand{\L}{\underline{L}}

\DeclareSymbolFont{boldoperators}{OT1}{cmr}{bx}{n}
\SetSymbolFont{boldoperators}{bold}{OT1}{cmr}{bx}{n}
\usepackage{accents}
\newcommand\thickbar[1]{\accentset{\rule{.45em}{.6pt}}{#1}}
\renewcommand{\bar}{\thickbar}
\renewcommand{\a}{\mathbf{a}}
\renewcommand{\k}{\mathbf{k}}

\newcommand{\ahom}{\bar{\a}}
\newcommand{\bhom}{\bar{\mathbf{b}}}
\newcommand{\shom}{\bar{\mathbf{s}}}
\newcommand{\khom}{\bar{\mathbf{k}}}

\newcommand{\bfA}{\mathbf{A}}
\newcommand{\bfAhom}{\overline{\mathbf{A}}}

\newcommand{\Bring}{\mathring{\underline{B}}}
\newcommand{\Hring}{\mathring{\underline{H}}}

\makeatletter 
\newcommand{\negphantom}{\v@true\h@true\negph@nt} 
\newcommand{\neghphantom}{\v@false\h@true\negph@nt} 
\newcommand{\negph@nt}{\ifmmode\expandafter\mathpalette 
  \expandafter\mathnegph@nt\else\expandafter\makenegph@nt\fi} 
\newcommand{\makenegph@nt}[1]{%
  \setbox\z@\hbox{\color@begingroup#1\color@endgroup}\finnegph@nt} 
\newcommand{\finnegph@nt}{%
  \setbox\tw@\null 
  \ifv@ \ht\tw@\ht\z@\dp\tw@\dp\z@\fi \ifh@\wd\tw@-\wd\z@\fi\box\tw@} 
\newcommand{\mathnegph@nt}[2]{%
  \setbox\z@\hbox{$\m@th #1{#2}$}\finnegph@nt} 
\makeatother

\newcommand{\Hminusuls}[1]{\hat{\phantom{H}}\negphantom{H}\underline{H}^{#1}}

\def\Xint#1{\mathchoice
{\XXint\displaystyle\textstyle{#1}}%
{\XXint\textstyle\scriptstyle{#1}}%
{\XXint\scriptstyle\scriptscriptstyle{#1}}%
{\XXint\scriptscriptstyle\scriptscriptstyle{#1}}%
\!\int}
\def\XXint#1#2#3{{\setbox0=\hbox{$#1{#2#3}{\int}$}
\vcenter{\hbox{$#2#3$}}\kern-.5\wd0}}

\def\fint{\Xint-}

\newcommand{\avsum}{\mathop{\mathpalette\avsuminner\relax}\displaylimits}

\makeatletter
\newcommand\avsuminner[2]{%
  {\sbox0{$\m@th#1\sum$}%
   \vphantom{\usebox0}%
   \ooalign{%
     \hidewidth
     \smash{\,\rule[.23em]{8.8pt}{1.1pt} \relax}%
     \hidewidth\cr
     $\m@th#1\sum$\cr
   }%
  }%
}
\makeatother

\makeatletter
\newcommand\avsuminnerr[2]{%
  {\sbox0{$\m@th#1\sum$}%
   \vphantom{\usebox0}%
   \ooalign{%
     \hidewidth
     \smash{\,\rule[.23em]{6pt}{0.7pt} \relax}%
     \hidewidth\cr
     $\m@th#1\sum$\cr
   }%
  }%
}
\makeatother

\let\originalleft\left
\let\originalright\right
\renewcommand{\left}{\mathopen{}\mathclose\bgroup\originalleft}
\renewcommand{\right}{\aftergroup\egroup\originalright}

\newcommand{\innerbox}[2]{%
	\raisebox{0.0ex}{%
		\ooalign{%
		\raisebox{#2ex}{\scalebox{#1}{$\square$}}%
	}}
}
\newcommand{\cu}{%
  \mathbin{%
    \mathchoice
      {\innerbox{1.0}{-0.2}}%
      {\innerbox{1.0}{-0.2}}%
      {\innerbox{0.75}{-0.2}}%
      {\innerbox{0.8}{-0.2}}%
  }%
}

\newcommand{\indc}{{\boldsymbol{1}}}
\renewcommand{\hat}{\widehat}

\usepackage{titlesec}

\newcommand{\addperiod}[1]{#1.}
\titleformat{\section}
   {\centering\normalfont\Large}{\thesection.}{0.5em}{}
\titleformat{\subsection}[runin]
  {\normalfont\bfseries}
  {\thesubsection.}
  {0.5em}
  {\addperiod}
\titleformat{\subsubsection}[runin]
  {\normalfont\bfseries}
  {\thesubsubsection.}
  {0.5em}
  {\addperiod}
\titleformat*{\subsubsection}{\normalfont\itshape}
\titleformat*{\paragraph}{\bfseries}
\titleformat*{\subparagraph}{\large\bfseries}

\title{Stochastic homogenization of coarse‑grained elliptic equations}

\author{
Aidan Lau
\thanks{Courant Institute of Mathematical Sciences, New York University.
{\footnotesize \href{mailto:aidan.lau@nyu.edu}{aidan.lau@nyu.edu}.}
}
}

\date{May 11, 2026}

\usepackage[nottoc,notlot,notlof]{tocbibind}

\begin{document}

\maketitle

\begin{abstract}
We prove quenched stochastic homogenization for divergence-form elliptic equations, under the assumption that the coefficients are stationary, ergodic, integrable, and satisfy a coarse-grained ellipticity assumption. The ellipticity assumption requires that the coefficients remain bounded in a negative regularity sense on large scales. As a corollary, we recover a sufficient joint integrability condition on the symmetric and skew-symmetric parts of the coefficient field.
\end{abstract}

\section{Introduction}

\subsection{Statement of results} We consider the elliptic equation
\begin{equation}
\label{e.PDE}
-\nabla \cdot \a(\cdot)\nabla u = \nabla \cdot \mathbf{f} \quad \mbox{in } U\subset \mathbb{R}^d \quad \mbox{with} \quad u = g \mbox{ on } \partial U\,,
\end{equation}
where~$d\geq 2$ and the coefficient field~$\a:\mathbb{R}^d \to \mathbb{R}^{d\times d}$ is given by a probability measure~$\mathbb{P}$ which is stationary and ergodic. We assume that the coefficient field~$\a$ satisfies a weak, coarse-grained ellipticity condition, which measures the size of a coefficient field in a negative-regularity sense by discounting behaviour on small scales. This notion of ellipticity, introduced in~\cite{AK.HC}, permits degenerate or singular fields which, in particular, may not satisfy an~$L^p$-type ellipticity condition. By adapting the methods of~\cite{AK.HC} we present a qualitative homogenization result within the same framework. Our assumptions are stated fully in Section~\ref{ss.assumptions}. Throughout the paper we decompose~$\a = \s + \k$ into its symmetric part~$\s= \frac12(\a+\a^t)$ and skew-symmetric part~$\k=\frac12(\a-\a^t)$.

\begin{theorem}
\label{t.main}
Assume~$\mathbb{P}$ satisfies~$\Zd$-stationarity~\ref{a.stationarity}, ergodicity~\ref{a.ergodicity}, the minimal moment condition~\ref{a.finite.E}, and the coarse-grained ellipticity condition~\ref{a.ellipticity}. Let~$s$ and~$t$ be as in~\ref{a.ellipticity}. There exists a  matrix~$\ahom\in\mathbb{R}^{d\times d}$ satisfying, for~$\shom := \frac12(\ahom+\ahom^t)$ and~$\khom:=\frac12(\ahom-\ahom^t)$, the bounds
\begin{equation}
\label{e.ahom.bounds}
\mathbb{E}\biggl[\fint_{\cu_0} \s^{-1} \biggr]^{-1}  \leq \shom \leq \shom+ \khom^t \shom^{-1} \khom
\leq 
\mathbb{E} \biggl[ \fint_{\cu_0} ( \s+\k^t\s^{-1}\k ) \biggr]
\end{equation}
such that for any~$n\in \mathbb{Z}$,~$\alpha \in (\max\{s,t\},1)$, and $u\in H^{1+\alpha}(\cu_n)$, if for each~$\ep \in (0,1)$ we let~$u^\epsilon$ denote the unique solution of the Dirichlet problem,
\begin{equation}
\label{e.thm.Dirichlet}
\left\{
\begin{aligned}
& \nabla \cdot \a(\tfrac{\cdot}{\epsilon}) \nabla u^\epsilon = \nabla \cdot  \ahom\nabla u & \mbox{ in } & \cu_n \,, \\
& u^\epsilon = u & \mbox{ on } & \partial \cu_n\,,
\end{aligned}
\right.
\end{equation}
then
\begin{equation}
\mathbb{P}\biggl[ \limsup_{\epsilon\to 0} \bigl(\|\nabla u^\epsilon - \nabla u \|_{H^{-\alpha}(\cu_n)} + \|\a(\tfrac{\cdot}{\epsilon})\nabla u^\epsilon - \ahom \nabla u\|_{H^{-\alpha}(\cu_n)}\bigr) = 0 \biggr] = 1\,.
\label{e.homogenization.limit}
\end{equation}
\end{theorem}
We make three remarks about the theorem. First, here and throughout the paper we use the notation
\begin{equation}
\label{e.notation}
\fint_U f(x)dx \coloneqq \frac{1}{|U|}\int_U f(x)
,dx\,,  \quad \avsum_{a\in A} x_a \coloneqq \frac{1}{|A|}\sum_{a\in A} x_a \,,\quad \mbox{and} \quad\cu_n \coloneqq \biggl( -\frac{3^n}{2},\frac{3^n}{2}\biggr)^d\,,
\end{equation}
and use the Loewner partial ordering on symmetric matrices. Second, Theorem~\ref{t.main} implies homogenization in the~$L^2$ sense, because we can combine the estimate~$\| u^\epsilon -  u \|_{L^{2}(\cu_n)}  \lesssim \|\nabla u^\epsilon - \nabla u\|_{H^{-1}(\cu_n)}$ with~\eqref{e.homogenization.limit} and $\alpha < 1$ to obtain
\begin{equation}
\mathbb{P}\biggl[ \limsup_{\epsilon\to 0} 
\| u^\epsilon -  u \|_{L^{2}(\cu_n)}  = 0 \biggr] = 1\,.
\label{e.homogenization.limit.in.L2}
\end{equation}
Finally, it is a closely related problem to construct a process with generator~$\nabla \cdot \a \nabla$ and prove an invariance principle for this process. In order to do so, it is necessary to place additional mild assumptions on the coefficient field because control over the process on small scales is required. We believe that with some work  it should be possible to construct a Feller process with generator~$\nabla \cdot \a\nabla$, and for a Feller process convergence of the rescaled process to a Brownian motion is equivalent to convergence of the generators~\cite[Theorem 17.25]{Kallenberg}.\footnote{Construction of a Feller process requires~$L^\infty$-type estimates, which are possible in the coarse-grained elliptic framework. In particular, we mention that local boundedness and a Harnack inequality for coarse-grained elliptic coefficient fields are proved in~\cite{AADKM}.} We do not consider these results in this paper and refer the reader instead to similar arguments in~\cite[Section 9]{ABK.SD}, or an alternative approach via Dirichlet forms theory in~\cite{CD2016}.

We now describe our assumptions briefly. Assumption~\ref{a.finite.E} is a minimal integrability condition ensuring that the quantities appearing in~\eqref{e.ahom.bounds} are positive and finite. The coarse-grained ellipticity condition~\ref{a.ellipticity} is the focus of this paper, and we provide an important sufficient condition in Corollary~\ref{c.Besov} below. This is a negative-regularity Besov-type condition which makes explicit our description of scale-discounted spatial averages of the field. We first define, for~$s\in (0,1)$ and~$n\in\mathbb{N}$,
\begin{equation}
\label{e.Binfty}
\|f\|_{\Bring_{\infty,1}^{-2s}(\cu_n)} := \sum_{k=-\infty}^n 3^{2s(k-n)} \max_{z\in 3^k\mathbb{Z}^d \cap \cu_n} \biggl| \fint_{z+\cu_k} f \biggr|\,.
\end{equation}

\begin{corollary}
\label{c.Besov}
Suppose that~$\mathbb{P}$ satisfies~\ref{a.stationarity},~\ref{a.ergodicity},~\ref{a.finite.E} and there exist~$s,t>0$ with~$s+t < 1$ such that
\begin{equation}
\label{e.Binfty.finite}
\mathbb{P}\biggl[\limsup_{n\to\infty } \|\s+\k^t\s^{-1}\k\|_{\Bring_{\infty,1}^{-2t}(\cu_n)} < \infty \quad \mbox{and} \quad \limsup_{n\to \infty } \|\s^{-1}\|_{\Bring_{\infty,1}^{-2s}(\cu_n)} < \infty \biggr] = 1\,.
\end{equation}
Then the conclusion of Theorem~\ref{t.main} holds.
\end{corollary}

The proof of Corollary~\ref{c.Besov} is given in~\eqref{e.Binfty.bound} below. The quantity defined in~\eqref{e.Binfty} is actually an upper bound for the corresponding Besov norm if~$s<\nicefrac12$, and for more detail on Besov norms in triadic cubes we refer to~\cite[Section 2.4]{AK.HC}. Although the norm in~\eqref{e.Binfty} may be unfamiliar, we can recover a more familiar integrability condition by embedding. If~$p > \nicefrac{d}{2t}$ and~$q > \nicefrac{d}{2s}$ then
\begin{equation}
\|\s+\k^t\s^{-1}\k\|_{\Bring_{\infty,1}^{-2t}(\cu_n)} \lesssim \|\s+\k^t\s^{-1}\k\|_{\L^p(\cu_n)} \qquad \mbox{and} \qquad \|\s^{-1}\|_{\Bring_{\infty,1}^{-2s}(\cu_n)} \lesssim \|\s^{-1}\|_{\L^q(\cu_n)}\,,
\end{equation}
and consequently~\ref{a.ellipticity} is satisfied if~$\mathbb{E}[\|\s+\k^t\s^{-1}\k\|^p_{\L^p(\cu_0)}] < \infty$ and~$\mathbb{E}[\|\s^{-1}\|^q_{\L^q(\cu_0)}] < \infty$ with
\begin{equation}
\label{e.moment.condition}
\frac{1}{p} + \frac{1}{q} < \frac{2}{d}\,.
\end{equation}
After stating our assumptions formally in Section~\ref{ss.assumptions}, we give an example of a specific field satisfying the assumptions of Corollary~\ref{c.Besov} and give the short proof of the~$L^p$ condition, which we state formally as Corollary~\ref{c.integrability} for comparison.

Moment conditions have been a focus of the existing literature as a natural extension of classical homogenization results (see~\cite{PV1981,PV1982,Osada,Kozlov79,Kozlov85} and the references therein) to unbounded or non-uniformly elliptic coefficient fields. In the case of a symmetric coefficient field, the model is often placed on the lattice and referred to as the random conductance model. On the other hand, given a uniformly elliptic environment, the skew-symmetric part is the stream matrix for a divergence-free, mean-zero flow. In either of these two cases Corollary~\ref{c.integrability} is not new, and not optimal in terms of integrability (see discussion below). However, other than~\cite{AK.HC} we do not believe that a condition allowing both non-uniformly elliptic~$\s$ and unbounded~$\k$ exists so far in the literature.

\begin{corollary}
\label{c.integrability}
Assume~\ref{a.stationarity} and~\ref{a.ergodicity} and suppose that there exist~$1\leq p,q \leq \infty$ with~$\frac{1}{p} + \frac{1}{q} < \frac{2}{d}$ such that
\begin{equation}
\label{e.integral.ellipticity}
\mathbb{E}\bigl[\|\s+\k^t\s^{-1}\k\|^p_{\L^p(\cu_0)}\bigr] < \infty \quad \mbox{and} \quad \mathbb{E}\bigl[\|\s^{-1}\|^q_{\L^q(\cu_0)}\bigr] < \infty\,.
\end{equation}
Then the conclusion of Theorem~\ref{t.main} holds.
\end{corollary}

Moment conditions on the coefficient field have a long history. Assuming a uniformly elliptic symmetric part, Oelschl\"ager~\cite{Oelschlager} proved an invariance principle assuming an~$L^2$-integrable stream matrix satisfying pathwise regularity and growth conditions. Avellaneda and Majda~\cite{AM1991} proved homogenization of the parabolic problem assuming~$L^p$-integrability of the stream matrix for~$p \geq \max\{d,2+\delta\}$. Assuming mainly~$p>d$ integrability and pathwise~$C^2$-smoothness of the stream matrix, Fannjiang and Komorowski~\cite{FK1997} then proved a quenched invariance principle for the associated process. It is natural to consider the case~$p=2$ because only the second moment of~$\k$ appears explicitly in bounding the homogenized matrix in~\eqref{e.ahom.bounds}, and under this assumption Kozma and T{\'o}th~\cite{KT2017} proved an invariance principle for the process in probability with respect to the environment. T{\'o}th~\cite{Toth2018} then proved a quenched central limit theorem assuming~$L^{2+\delta}$-integrability of the stream matrix. Recently, Fehrman~\cite{Ben2023} used PDE-based methods to prove quenched homogenization and large-scale regularity under the~$p \geq \max\{d,2+\delta\}$ assumption, as well as quenched homogenization under the~$L^2$-integrability assumption.

Applying Corollary~\ref{c.integrability} with~$q=\infty$ we obtain quenched homogenization of solutions to the Dirichlet problem assuming~$L^p$-integrability of the stream matrix for~$p > d$. As noted in~\cite{Ben2023}, this is the threshold which implies convergence of the two-scale expansion and large-scale regularity. Since large-scale regularity results can be obtained in the coarse-grained elliptic framework (for example see~\cite[Section 5.2]{AK.HC}), we believe it is consistent that our results require the~$p>d$ assumption. However, we emphasize that we do not recover the homogenization results of~\cite{KT2017,Toth2018,Ben2023} which require only~$p = 2$ or~$2+\delta$.

For the random conductance model in two dimensions, Biskup~\cite{Biskup2011} proved a quenched invariance principle for the process assuming only finite expectation of~$\s$ and~$\s^{-1}$, which is optimal. For~$d\geq 2$, Andres, Deuschel and Slowik~\cite{ADS2015} proved a quenched invariance principle under the assumptions~$\s\in L^p$ and~$\s^{-1}\in L^q$ with~$p$ and~$q$ satisfying~\eqref{e.moment.condition}. Chiarini and Deuschel~\cite{CD2016} proved the same result in the continuum, defining the process by Dirichlet forms theory. Bella and Sch\"affner~\cite{BellaSchaf2020} later improved the result for~$d\geq 3$ in the discrete setting by proving a quenched invariance principle for the process with
\begin{equation*}
\frac{1}{p} + \frac{1}{q} < \frac{2}{d-1}\,,
\end{equation*}
which is optimal for~$L^\infty$-sublinearity of the corrector. The conclusion of Corollary~\ref{c.integrability} is strictly weaker than the results of~\cite{BellaSchaf2020} and~\cite{Biskup2011}. However, we note that the goal of this paper is in some sense parallel: instead of optimizing in integrability we provide here an alternative criterion for homogenization, and recover an integrability criterion as a consequence.

The notion of coarse-grained ellipticity which we use in this paper was first introduced by Armstrong and Kuusi~\cite{AK.HC}. Although they proved quantitative homogenization bounds under quantitative assumptions, the ellipticity assumption in that paper is closely related to~\ref{a.ellipticity}. However, we note that while we allow the parameters~$s+t < 1$ in~\ref{a.ellipticity}, the assumptions in~\cite{AK.HC} restrict to the case~$s=t < \nicefrac12$. Our methods are an adaptation of those of~\cite{AK.HC}, but our arguments are considerably simpler.

In the rest of this introduction we state the assumptions precisely and prove Corollary~\ref{c.Besov} and Corollary~\ref{c.integrability}. In Section~\ref{s.coarse.graining} we define the coarse-grained matrices and make precise the sense in which we solve the Dirichlet problem in~\eqref{e.thm.Dirichlet}. In Section~\ref{s.ellip.converge} we prove that the coarse-grained matrices converge, and use this to prove our main theorem in Section~\ref{s.homogenize}.

\subsection{Statement of assumptions and examples}
\label{ss.assumptions}

Our probability space is the set of coefficient fields for which the equation~\eqref{e.PDE} is well-posed. Given~$\a = \s+\k$, where~$\s$ is symmetric and~$\k$ is skew-symmetric, we assume that the symmetric part satisfies
\begin{equation}
\label{e.sym.assumption}
\s,\s^{-1}\in L^1_{\mathrm{loc}}(\mathbb{R}^d;\mathbb{R}^{d\times d})\,,
\end{equation}
and the skew-symmetric part satisfies either
\begin{equation}
\label{e.skew.assumption}
\s^{-\nicefrac12}\k\s^{-\nicefrac12} \in L^\infty_{\mathrm{loc}}(\mathbb{R}^d;\mathbb{R}^{d\times d}) \qquad \mbox{or} \qquad \s^{-\nicefrac12}\k\,, \, \s^{-\nicefrac12}\nabla \cdot \k \in L^2_{\mathrm{loc}}(\mathbb{R}^d;\mathbb{R}^{d\times d})\,.
\end{equation}
We define the space
\begin{equation}
\Omega \coloneqq \bigl\{ \a:\mathbb{R}^d \to \mathbb{R}^{d\times d} \quad \mbox{satisfying } \eqref{e.sym.assumption} \mbox{ and } \eqref{e.skew.assumption} \bigr\}\,,
\label{e.Omega.def}
\end{equation}
and equip~$\Omega$ with a~$\sigma$-field by defining, for every Borel subset~$U \subseteq \mathbb{R}^d$, the~$\sigma$-field~$\mathcal{F}(U)$ generated by the family of random variables
\begin{equation*}
\a \mapsto \int_U e' \cdot \a(x) e \varphi (x) \, dx\,, \quad e,e'\in\mathbb{R}^d, \varphi \in C^\infty_c(U)\,.
\end{equation*}
We also set~$\mathcal{F}\coloneqq \mathcal{F}(\mathbb{R}^d)$. The group of~$\mathbb{R}^d$ translations is denoted~$\{T_z:z\in\mathbb{R}^d\}$, where~$T_z:\Omega \to \Omega$ acts by~$T_z : \a(\cdot) \to \a(\cdot + z)$. 
It is convenient to define
\begin{equation}
\label{e.b.def}
\b\coloneqq \s + \k^t\s^{-1}\k
\end{equation}
since~$\b$ plays the role of an ellipticity upper bound, while~$\s^{-1}$ plays the role of an ellipticity lower bound. For parameters~$s,t \in (0,1)$ and~$n\in\mathbb{N}$ we define the coarse-grained ellipticity constants in~$\cu_n$ (recall the notation from~\eqref{e.notation}) by
\begin{align}
\label{e.lambda.def}
\lambda_{s}(\cu_n) & \coloneqq \left((1-3^{-2s})\sum_{k=-\infty}^{n} 3^{2s(k-n)} \max_{z\in 3^{k}\mathbb{Z}^d \cap \cu_n} |\s_*^{-1}(z+\cu_k) |\right)^{-1}
\end{align}
and
\begin{align}
\label{e.Lambda.def}
\Lambda_{t}(\cu_n) & \coloneqq (1-3^{-2t})\sum_{k=-\infty}^{n} 3^{2t(k-n)} \max_{z\in 3^{k}\mathbb{Z}^d \cap \cu_n} |\b(z+\cu_k)| \,.
\end{align}
In the definitions above we take the matrix norm to be the spectral norm. The normalizing constants in front are chosen so that~$(1-3^{-2t})\sum_{k=-\infty}^n 3^{2t(k-n)} = 1$, which implies by Proposition~\ref{p.cg.mat} that
\begin{equation}
\|\s^{-1}\|_{L^\infty(\cu_n)}^{-1} \leq \lambda_s(\cu_n)  \quad \mbox{and} \quad \Lambda_t(\cu_n) \leq \|\b\|_{L^\infty(\cu_n)}\,.
\end{equation}
It is conceptually consistent that the coarse-grained ellipticity constants are bounded by the uniform ellipticity constants, but in calculations it is much easier to work without the normalizing constants. For this reason we will throughout the paper ignore the normalizing constants in~\eqref{e.lambda.def} and~\eqref{e.Lambda.def} and let most constants depend on~$s$ and~$t$. Finally, we extend the definitions~\eqref{e.lambda.def} and~\eqref{e.Lambda.def} to domains~$\cu_n$ with~$n\in\mathbb{R}$ in the natural way by taking a triadic decomposition at scales~$k$ such that~$3^k = 3^{n+j}$ for~$j \in (-\infty, 0] \cap \mathbb{Z}$.

We can now state our assumptions fully. We assume that the law of the coefficient field~$\a$ is given by a probability measure~$\mathbb{P}$ on~$(\Omega,\mathcal{F})$ which satisfies:

\begin{enumerate}[label=(\textrm{P\arabic*})]
\setcounter{enumi}{0}
\item \emph{Stationarity with respect to~$\mathbb{Z}^d$ translations:}
\label{a.stationarity}
\begin{equation*}
\mathbb{P} = \mathbb{P} \circ T_z\,, \forall z\in\mathbb{Z}^d\,.
\end{equation*}
\end{enumerate}

\begin{enumerate}[label=(\textrm{P\arabic*})]
\setcounter{enumi}{1}
\item \emph{Ergodicity with respect to~$\mathbb{Z}^d$ translations:}
\label{a.ergodicity}
For all~$A\in\mathcal{F}$, if~$A = \bigcap_{z\in \mathbb{Z}^d}T_zA$ then~$\mathbb{P}[A] \in \{0,1\}$.
\end{enumerate}

\begin{enumerate}[label=(\textrm{P\arabic*})]
\setcounter{enumi}{2}
\item \emph{Finite first moments:} We have
\begin{equation}
\label{e.in.L1}
 \mathbb{E}\biggl[ \fint_{\cu_0}\s^{-1}\biggr] < \infty \quad \mbox{and} \quad \mathbb{E}\biggl[ \fint_{\cu_0} \s + \k^t\s^{-1}\k\biggr] < \infty\,.
\end{equation}
\label{a.finite.E}
\end{enumerate}
\begin{enumerate}[label=(\textrm{P\arabic*})]
\setcounter{enumi}{3}
\item \emph{Finite coarse-grained ellipticity:}
There exist~$s,t>0$ with~$s+t<1$ such that 
\begin{equation}
\label{e.lambdas.finite}
\P \biggl[ 
\limsup_{n\to\infty} 
\lambda^{-1}_{s}(\cu_n) < \infty 
\quad\mbox{and} \quad
\limsup_{n\to\infty} 
\Lambda_{t}(\cu_n) < \infty
\biggr] 
= 1\,.
\end{equation}
\label{a.ellipticity}
\end{enumerate}

We will illustrate the coarse-grained ellipticity assumption by proving Corollary~\ref{c.Besov} and Corollary~\ref{c.integrability}, and giving an example of a field satisfying the Besov-type condition but not the~$L^p$ condition. 

We first give the short proof of Corollary~\ref{c.Besov}. By Proposition~\ref{p.cg.mat} we have, for any domain~$U$, that~$\b(U) \leq \fint_U \b$ and~$\s_*^{-1}(U) \leq \fint_U \s^{-1}$. Combining this with the definition of the coarse-grained ellipticity in~\eqref{e.Lambda.def} we obtain
\begin{equation}
\label{e.Binfty.bound}
\Lambda_t(\cu_n) \leq (1-3^{-2t})\sum_{k=-\infty}^n 3^{2t(k-n)} \max_{z\in 3^k\mathbb{Z}^d \cap \cu_n} \biggl| \fint_{z+\cu_k}\b\biggr|\,,
\end{equation}
and similarly for~$\lambda_s^{-1}(\cu_n)$. The right-hand side of~\eqref{e.Binfty.bound} is just the norm in~\eqref{e.Binfty}, so~\eqref{e.Binfty.finite} immediately implies~\ref{a.ellipticity}.

To prove Corollary~\ref{c.integrability} we first take~\eqref{e.Binfty.bound}, control the maximum over sub-cubes ($\ell^\infty$ norm) by the~$\ell^p$ norm for~$p > \nicefrac{d}{2t}$, correct for the averaging factor in the sum, and bound the spatial average by the~$L^p$ norm to get
\begin{equation}
\label{e.how.to.bound}
\Lambda_{t}(\cu_n) \leq (1-3^{-2t})\sum_{k=-\infty}^{n} 3^{(2t-\nicefrac{d}{p})(k-n)} \biggl(\avsum_{z\in 3^{k}\mathbb{Z}^d \cap \cu_n} \biggl|\fint_{z+\cu_k} \b \biggr|^p \biggr)^{\nicefrac{1}{p}} \leq \frac{1-3^{-2t}}{1-3^{\nicefrac{d}{p}-2t}} \|\b\|_{\L^p(\cu_n)}\,.
\end{equation}
By the same reasoning, if~$q > \nicefrac{d}{2s}$ we can bound
\begin{equation}
\lambda_s^{-1}(\cu_n) \leq \frac{1-3^{-2s}}{1-3^{\nicefrac{d}{q}-2s}} \|\s^{-1}\|_{\L^q(\cu_n)}\,.
\end{equation}
Now by the ergodic theorem
\begin{equation}
\label{e.Lp.finite}
\lim_{n\to\infty}\|\b\|^p_{\L^p(\cu_n)} = \mathbb{E}\bigl[\|\b\|_{\L^p(\cu_0)}^p \bigr] \quad \mbox{and} \quad \lim_{n\to\infty}\|\s^{-1}\|^q_{\L^q(\cu_n)} = \mathbb{E}\bigl[\|\s^{-1}\|_{\L^q(\cu_0)}^q \bigr]
\end{equation}
when the expectations are finite. If~$\nicefrac{1}{p}+\nicefrac{1}{q} < \nicefrac{2}{d}$ then there exist~$s,t>0$ such that~$s+t<1$, $p> \nicefrac{d}{2t}$ and~$q>\nicefrac{d}{2s}$. This implies that~\ref{a.finite.E} and~\ref{a.ellipticity} are satisfied if the expectations on the right-hand side of~\eqref{e.Lp.finite} are finite, which proves Corollary~\ref{c.integrability}.

The Besov-type condition of Corollary~\ref{c.Besov} is satisfied for many fields which may not satisfy the assumptions of Corollary~\ref{c.integrability}. We construct such a field here, based on the idea of iterated multiplication of independent fields, as in~\cite{KahanePeyriere1976, Kahane1985}. The construction we present can be generalized to Gaussian fields that admit a similar type of scale decomposition and are not piece-wise constant.

Suppose that for each~$j\in \mathbb{N}$ we let~$\varphi_j(\cdot)$ be a~$\mathbb{Z}^d$ stationary, Gaussian field that is piecewise constant on the cubes~$z+\cu_j$ for~$z\in 3^j\mathbb{Z}^d$, and on each cube is a centred Gaussian random variable with variance~$\sigma^2$. Suppose that the~$\varphi_j$ are independent, and let~$W_j(\cdot) = e^{\varphi_j(\cdot) - \nicefrac{\sigma^2}{2}}$ so that for all~$x$,~$\mathbb{E}[W_j(x)] = 1$ and~$W_j$ is a positive random field. If we define
\begin{equation*}
f_m(\cdot) = \prod_{j=1}^m W_j(\cdot)\,,
\end{equation*}
then by independence
\begin{equation*}
\mathbb{E}\bigl[|f_m(\cdot)|^p\bigr] = \prod_{j=1}^m \mathbb{E}[ |W_j(\cdot)|^p] = \exp\Bigl( \frac {mp(p-1)\sigma^2}{2} \Bigr)\,,
\end{equation*}
so that the~$L^p$ norm of~$f_m(\cdot)$ grows exponentially in~$m$ for any~$p > 1$. However, we do have a uniform in~$m$ control over~$f_m(\cdot)$ in the Besov-type norm~\eqref{e.Binfty}. These two facts allow us to construct a field which is bounded in the Besov-type norm but not in~$L^p$ for~$p>1$, by defining
\begin{equation}
\label{e.our.field}
f(\cdot) = \sum_{m=1}^\infty \frac{f_m(\cdot)}{m^3}\,,
\end{equation}
which converges in~$L^1(\Omega; L^1_{\mathrm{loc}}(\mathbb{R}^d))$ and is a~$\mathbb{Z}^d$-stationary, ergodic random field.

\begin{example}
\label{ex.field}
If
\begin{equation}
\label{e.regularity.condition}
\sigma < \min\bigg\{1, \sqrt{\frac{2}{d}}t \bigg\}
\end{equation}
then
\begin{equation}
\label{e.field.in.Besov}
\mathbb{P}\biggl[ \limsup_{n\to\infty} \|f\|_{\Bring_{\infty,1}^{-2t}(\cu_n)} < \infty \biggr] = 1\,.
\end{equation}
\end{example}
We prove this result in Appendix~\ref{appendix.ex}, and note that we have made no attempt to optimize the choice of parameters. Since~$\sigma$ and~$t$ can be chosen to satisfy~\eqref{e.regularity.condition} and~$t<1$, if we let~$s = \frac{1-t}{2}$ and define
\begin{equation}
\a(\cdot) := (1+f(\cdot))\mathrm{I}_d
\end{equation}
then we obtain a field satisfying the assumptions~\ref{a.finite.E} and~\ref{a.ellipticity} which is not in~$L^p(\Omega, L^p_{\mathrm{loc}}(\mathbb{R}^d))$ for any~$p >1$.

\section{Coarse-grained matrices}
\label{s.coarse.graining}

\subsection{Definition of the coarse-grained matrices}
\label{ss.cg.defs}
In this section we prove that under the assumption~$\a\in \Omega$ the PDE~\eqref{e.PDE} is well-posed and use this to define the coarse-grained coefficients. The arguments in this section are entirely deterministic. We consider coefficient fields~$\a :\mathbb{R}^d \to \mathbb{R}^{d\times d}$ with symmetric part and skew-symmetric part given by
\begin{equation}
\label{e.def.sk}
\s(x) \coloneqq \frac{\a(x) + \a^t(x)}{2} \quad \mbox{and} \quad \k(x) \coloneqq \frac{\a(x) - \a^t(x)}{2}\,,
\end{equation}
where~$A^t$ denotes the transpose of a matrix~$A$. Since~$\a$ is elliptic, the symmetric part~$\s$ is invertible. Proceeding as in~\cite{AK.HC} our goal is to define, for bounded Lipschitz domains~$U$ and~$p,q\in\mathbb{R}^d$,
\begin{equation*}
J(U,p,q) = \sup_{-\nabla \cdot \a\nabla u = 0} \fint_U
\Bigl( 
-\frac{1}{2}\nabla u \cdot \s \nabla u - p \cdot \a\nabla u + q\cdot \nabla u
\Bigr)
\,.
\end{equation*}
In order to do so we need to make sense of the PDE~\eqref{e.PDE} under the assumption~$\a\in\Omega$. We begin by introducing the weighted Sobolev space~$H^1_{\s}(U)$ to be the set of weakly differentiable functions such that
\begin{equation}
\label{e.H1s.def}
\|u\|_{H^1_{\s}(U)} \coloneqq \biggl(\biggl(\int_U |u| \biggr)^2 + \int_U \nabla u \cdot \s\nabla u \biggr)^{\nicefrac12} < \infty \,.
\end{equation}
This defines a Banach space (see~\cite{KO84}) in which smooth functions are dense, and there is a linear, bounded trace operator~$\mathrm{Tr}: H^1_{\s}(U) \to L^1(\partial U)$ defined by a straightforward adaptation of the standard proof, for instance as in~\cite[Chapter 5.5]{Evans}. We can then define~$H^1_{\s,0}(U)$ to be the subspace of trace-zero functions, which is a Hilbert space under the norm
\begin{equation}
\|u\|_{\dot{H}_{\s,0}^1(U)} = \biggl( \int_U \nabla u \cdot \s \nabla u\biggr)^{\nicefrac12}\,.
\end{equation}
Again by modifying the standard argument we also see that~$H^1_{\s,0}(U)$ is equal to the closure of~$C_c^\infty(U)$ under the~$H^1_{\s}(U)$ norm. We define~$H^{-1}_{\s}(U)$ as the dual of~$H^1_{\s,0}(U)$ with norm
\begin{equation}
\label{e.Hdual.def}
\|f\|_{H^{-1}_{\s}(U)} \coloneqq \sup \bigl\{\langle f,g\rangle : \|g\|_{H^1_{\s,0}(U)} \leq 1 \bigr\}\,.
\end{equation}
Applying the Riesz representation theorem, for any~$f\in H^{-1}_{\s}(U)$ the problem
\begin{equation}
\label{e.sym.Dirichlet}
\begin{cases}
-\nabla \cdot \s\nabla u = f & \mbox{in } U\,,\\
u = 0 & \mbox{on } \partial U\,,
\end{cases}
\end{equation}
then has a unique solution~$u\in H^1_{\s,0}(U)$, which allows us to characterize the dual space by
\begin{equation}
\label{e.characterize.Hdual}
H^{-1}_{\s}(U) = \{ \nabla \cdot \s^{\nicefrac12}\mathbf{f} : \mathbf{f} \in L^2(U)^d\}\,.
\end{equation}

If we consider the full problem
\begin{equation}
\label{e.Dirichlet}
\begin{cases}
-\nabla \cdot (\s+\k)\nabla u = \nabla \cdot \s^{\nicefrac12}\mathbf{f} & \mbox{in } U\,, \\
u = 0 & \mbox{on } \partial U\,,
\end{cases}
\end{equation}
then we obtain the existence of a solution~$u\in H^1_{\s,0}(U)$ for any~$\mathbf{f}\in L^2(U)^d$ by the Lions-Lax-Milgram lemma~\cite[Theorem 2.1]{Showalter}. We apply the theorem to the bilinear form
\begin{equation*}
B : \dot{H}_{\s,0}^1(U) \times V \to \mathbb{R} \quad \mbox{given by} \quad B[u,\varphi] = \int_U \nabla u \cdot \a \nabla \varphi\,,
\end{equation*}
where the normed vector space~$V$ is the set of functions~$\varphi \in C_c^\infty$, equipped with the~$\dot{H}_{\s,0}^1(U)$ norm. We have two possible conditions on the skew-symmetric part~$\k$, both of which allow us to prove uniqueness of solutions. If~$\s^{-\nicefrac12}\k\s^{-\nicefrac12}\in L^\infty_{\mathrm{loc}}$ then we can test the equation with the solution, in which case the basic energy estimate implies uniqueness. In the case that~$\s^{-\nicefrac12}\k,\s^{-\nicefrac12}\nabla \cdot \k \in L^2(U)$ we argue as in~\cite[Proposition 2.4]{Ben2023}, treating the skew-symmetric matrix~$\k$ as the stream matrix for an~$L^2$ integrable drift term. Using the skew-symmetry of~$\k$, for any~$\varphi \in C_c^\infty(U)$
\begin{equation*}
\int_U \nabla u \cdot \k \nabla \varphi = -\int_U \s^{-\nicefrac12}(\nabla \cdot \k) \cdot \s^{\nicefrac12}\nabla u \, \varphi\,,
\end{equation*}
so we can take~$\varphi$ to be mollifications of~$u_n \coloneqq (u\wedge n) \vee (-n)$ and pass to the limit by the dominated convergence theorem. This implies that if~$u$ solves~\eqref{e.Dirichlet} with~$\mathbf{f}=0$ then for all~$n$
\begin{equation*}
\fint_U \nabla u \cdot \s \nabla u_n = \fint_U \s^{-\nicefrac12}(\nabla \cdot \k) \cdot \nabla u \, u_n = \fint_U \s^{-\nicefrac12}(\nabla \cdot \k) \cdot \s^{\nicefrac12}\nabla (uu_n - \frac{1}{2}u_n^2) = 0\,.
\end{equation*}
Passing to the limit~$n\to\infty$ we conclude that~$u = 0$.

Finally, we note that we can solve the Dirichlet problem with boundary condition~$g\in W^{1,\infty}(U)$ and zero right-hand side by solving~\eqref{e.Dirichlet} with right-hand side~$\nabla \cdot \a\nabla g \in H^{-1}_{\s}(U)$ to obtain a function~$u$ such that~$-\nabla \cdot \a\nabla (u+g) = 0$. We remark that implicit in the above discussion are constants that depend on the norms of our coefficient field~$\a$. Although we need this qualitative framework to define the coarse-grained coefficients, in Section~\ref{ss.cg.analysis} we show that this dependence can be made explicit and depends only on the coarse-grained ellipticity constants~$\lambda_s(\cu_n)$ and~$\Lambda_t(\cu_n)$.

\smallskip

The above discussion shows that the Dirichlet problem is well-posed, and we note that the Neumann problem is well-posed by similar arguments. We now return to our definition of~$J(U,p,q)$. We will maximize over the (non-empty) solution space
\begin{equation}
\label{e.mathcalA.def}
\mathcal{A}(U) \coloneqq \bigl\{ u \in H^1_{\s}(U) : \nabla \cdot \a \nabla u = 0\,, (u)_U = 0\bigr\}\,,
\end{equation}
where here and throughout the paper we define~$(u)_U = \fint_U u$.
The space~$\mathcal{A}(U)$ is closed under the~$\dot{H}^1_{\s}(U)$ norm, so for any~$p,q\in\mathbb{R}^d$ and bounded Lipschitz domain~$U$ we apply~\cite[Chapter II.1]{Temam} to conclude that the functional
\begin{equation}
\label{e.J.def}
J(U,p,q) \coloneqq \sup_{u\in\mathcal{A}(U)} \fint_U \biggl( -\frac{1}{2}\nabla u \cdot \s \nabla u - p \cdot \a\nabla u + q\cdot \nabla u \biggr)
\end{equation}
is well-defined and has a unique maximizer~$v(U,p,q)$. Given the well-posedness of the variational problem in~\eqref{e.J.def}, standard arguments (see \cite{AK.Book}) imply that the coarse-grained coefficients~$\s(U),\s_*^{-1}(U)$ and~$\k(U)$ can be defined such that
\begin{equation}
\label{e.Jmat}
J(U,p,q) = \frac{1}{2}p \cdot \s(U) p + \frac{1}{2}(q+\k(U)p)\cdot \s_*^{-1}(U)(q + \k(U)p) - p\cdot q\,.
\end{equation}
We define the double-variable matrix
\begin{equation}
\label{e.bfA.def}
\bfA(U) \coloneqq \begin{pmatrix}
\s(U) + \k^t(U)\s_*^{-1}(U) \k(U) & -\k^t(U)\s_*^{-1}(U) \\
-\s_*^{-1}(U)\k(U) & \s_*^{-1}(U)\,,
\end{pmatrix}
\end{equation}
and
\begin{equation}
\label{e.bU.def}
\b(U) \coloneqq \s(U) + \k^t(U)\s_*^{-1}(U)\k(U)\,.
\end{equation}
The double-variable matrix should be viewed as a coarse-graining of
\begin{equation}
\label{e.A.def}
\bfA : = \begin{pmatrix}
\s + \k^t\s^{-1}\k & -\k\s^{-1} \\ -\s^{-1}\k & \s^{-1}
\end{pmatrix}\,.
\end{equation}
To summarize the above discussion we state it formally as a proposition and include the standard coarse-graining properties for future reference.

\begin{proposition}
\label{p.cg.mat}
Suppose that~$\a\in\Omega$. Then for any bounded Lipschitz domain~$U$ and~$p,q\in\mathbb{R}^d$,~$J(U,p,q)$ in~\eqref{e.J.def} is well-defined, the coarse-grained matrices~$\s(U),\s_*^{-1}(U)$ and~$\k(U)$ are defined and satisfy~\eqref{e.Jmat}, and the following holds:
\begin{itemize}
\item[•] The mean-zero maximizer~$v(U,p,q)$ in~\eqref{e.J.def} exists, is unique, and
\begin{equation}
\label{e.J.energy}
J(U,p,q) = \frac{1}{2}\fint_U \nabla v(U,p,q)\cdot \s \nabla v(U,p,q)\,.
\end{equation}
\item[•]
The coarse-grained matrices satisfy the ordering
\begin{equation}
\label{e.cg.mat.bound}
\biggl(\fint_U \s^{-1}\biggr)^{-1} \leq \s_*(U) \leq \s(U) \leq  \b(U) \leq \fint_U (\s+\k^t\s^{-1}\k)\,.
\end{equation}
\item[•] For any~$w\in \mathcal{A}(U)$ we have the coarse-graining inequalities
\begin{align}
\label{e.cg.ineq}
\left\{
\begin{aligned}
& \biggl( \fint_U \nabla w\biggr) \cdot \s_*(U) \biggl( \fint_U \nabla w\biggr)  \leq \fint_U \nabla w \cdot \s \nabla w\,, \\
& \biggl( \fint_U \a\nabla w\biggr) \cdot \b^{-1}(U) \biggl( \fint_U \a\nabla w\biggr)  \leq \fint_U \nabla w \cdot \s \nabla w \,. \\
\end{aligned}
\right.
\end{align}
and
\begin{equation}
\label{e.cg.diff}
\biggl| \fint_U \bigl( p \cdot \a\nabla w - q\cdot \nabla w\bigr)\biggr| = \biggl| \fint_U \nabla w \cdot \s \nabla v(U,p,q)\biggr| \leq (2J(U,p,q))^{\nicefrac12} \biggl( \fint_U \nabla w \cdot \s \nabla w\biggr)^{\nicefrac12}\,.
\end{equation}
\item[•] For any~$w\in\mathcal{A}(U)$,
\begin{align}
\label{e.quad.response}
J(U,p,q) - \fint_U & \biggl( -\frac{1}{2}\nabla w \cdot \s \nabla w - p \cdot \a\nabla w + q\cdot \nabla w \biggr) \nonumber \\
& = \fint_U \frac{1}{2}\bigl(\nabla v(U,p,q) - \nabla w\bigr)\cdot \s \bigl( \nabla v(U,p,q)-\nabla w\bigr) \,.
\end{align}
\item[•] Averages of the maximizers are given explicitly by
\begin{equation}
\label{e.averages}
\left\{
\begin{aligned}
&\fint_U \nabla v(U,p,q) = -p + \s_*^{-1}(U)(q+\k(U)p)\,, \\
& \fint_U \a\nabla v(U,p,q) = (\Id - \k^t(U)\s_*^{-1}(U))q - \b(U)p\,. \\
\end{aligned}
\right.
\end{equation}
\item[•] In terms of the double-variable matrix,
\begin{equation}
\label{e.J.bigA}
J(U,p,q) = \frac{1}{2}\begin{pmatrix}
-p \\ q
\end{pmatrix}
\cdot \bfA(U) \begin{pmatrix}
-p \\ q
\end{pmatrix}
-p\cdot q\,.
\end{equation}
\end{itemize}

\end{proposition}

We omit the proof of Proposition~\ref{p.cg.mat} because all the properties are essentially algebraic facts which follow from the variational representation of~$J(U,p,q)$ in~\eqref{e.J.def}. For the proofs of these facts, the proof of the following proposition, and more properties, we refer to~\cite[Chapter 5]{AK.Book}.

We will also need to consider the adjoint problem where~$\a$ is replaced with~$\a^t$. Since~$\a\in \Omega$ if and only if~$\a^t\in\Omega$ we can make sense of the problem~$-\nabla \cdot \a^t \nabla u = 0$ exactly as before, define the solution space~$\mathcal{A}_*(U)$ to be the set of solutions to the adjoint problem, and let
\begin{equation}
\label{e.Jstar.def}
J^*(U,p,q') \coloneqq \sup_{u^* \in \mathcal{A}_*(U)} \fint_U - \frac{1}{2}\nabla u^* \cdot \s \nabla u^* - p \cdot \a^t \nabla u^* + q'\cdot \nabla u^*\,.
\end{equation}

\begin{proposition}
Suppose that~$\a\in\Omega$. Then~$J^*(U,p,q')$ in~\eqref{e.Jstar.def} is well-defined for any bounded Lipschitz domain~$U$ and~$p,q'\in\mathbb{R}^d$, and:
\begin{itemize}
\item[•] The mean-zero maximizer~$v^*(U,p,q')$ in~\eqref{e.Jstar.def} exists, is unique, and
\begin{equation}
\label{e.Jstar.energy}
J^*(U,p,q') = \frac{1}{2}\fint_U \nabla v^*(U,p,q')\cdot \s\nabla v^*(U,p,q')\,.
\end{equation}
\item[•] In terms of the coarse-grained matrices,
\begin{equation}
\label{e.Jstar.mat}
J^*(U,p,q') = \frac{1}{2}p \cdot \s(U) p + \frac{1}{2}(q' - \k(U)p)\cdot \s_*^{-1}(U)(q' - \k(U)p) - p\cdot q'\,.
\end{equation}
That is, when we coarse-grain the adjoint problem we change~$\k(U) \to -\k(U)$.
\item[•] In terms of the double-variable matrix
\begin{align}
\label{e.J.bfA}
J^*(U,p,q') \nonumber = \frac{1}{2}\begin{pmatrix}
p \\ q'
\end{pmatrix}
\cdot \bfA(U) \begin{pmatrix}
p \\ q'
\end{pmatrix}
- p \cdot q'\,.
\end{align}
\item[•] The double variable matrix~$\bfA(U)$ is subadditive: if~$U_i$ are a partition of~$U$ then
\begin{equation}
\label{e.bfA.subadd}
\bfA(U) \leq \sum_{i=1}^n \frac{|U_i|}{|U|}\bfA(U_i)\,.
\end{equation}
\item[•] The double-variable matrix is uniformly elliptic at each scale:
\begin{equation}
\label{e.bfAs.ordered}
\biggl( \fint_U \bfA^{-1} \biggr)^{-1} \leq \bfA(U) \leq \fint_U \bfA\,.
\end{equation}
\end{itemize}

\end{proposition}

Finally, we note that the set of solutions to the equation~$-\nabla \cdot \a \nabla u = 0$ is unaffected by the addition to~$\a$ of a constant skew-symmetric matrix, and therefore we should only regard the problem as well-defined up to a centering operation. The coarse-graining commutes with this centering operation: if we fix a constant, skew-symmetric matrix~$\mathbf{h}$ and consider the coarse-graining of the operator~$\a - \mathbf{h}$, then the coarse-grained matrices~$\s(U)$ and~$\s_*(U)$ remain unchanged, while~$\k(U)$ is replaced by~$\k(U) - \mathbf{h}$. For more details we refer to~\cite[Section 2.5]{AK.HC}.

\subsection{Coarse-grained energy estimates}
\label{ss.cg.analysis}

In this section we introduce the function spaces we will use throughout the paper and prove energy estimates for solutions which depend on~$\a$ only through the coarse-grained ellipticity constants. We first introduce the function spaces. For every~$n\in\mathbb{Z}$ and~$s\in(0,1)$ define the semi-norm
\begin{equation}
\label{e.Besov.def}
[g]_{\underline{H}^s(\cu_n)} \coloneqq \biggl( \sum_{k=-\infty}^n  3^{-2sk} \avsum_{z\in 3^{k-1}\mathbb{Z}^d \cap \cu_n} \|g - (g)_{z+\cu_k\cap \cu_n}\|_{\L^2(z+\cu_k \cap \cu_n)}^2 \biggr)^{\nicefrac{1}{2}}\,,
\end{equation}
and norm
\begin{equation}
\label{e.Besov.norm.def}
\| g\|_{\underline{H}^s(\cu_n)} \coloneqq 3^{-sn}\|g\|_{\L^2(\cu_n)} + [g]_{\underline{H}^s(\cu_n)}\,.
\end{equation}
The space~$H^s(\cu_n)$ is defined as the set of functions~$g$ such that the norm in~\eqref{e.Besov.norm.def} is finite, or equivalently, the closure of smooth functions under that norm. By~\cite[Lemma A.4]{AK.HC} there exists a constant~$C(d)$ such that
\begin{equation}
\label{e.seminorm.integral}
\biggl( \fint_{\cu_n} \int_{\cu_n} \frac{|g(x)-g(y)|^2}{|x-y|^{d+2s}}\biggr)^{\nicefrac12} \leq C[g]_{\underline{H}^s(\cu_n)} \leq C\biggl( \fint_{\cu_n} \int_{\cu_n} \frac{|g(x)-g(y)|^2}{|x-y|^{d+2s}}\biggr)^{\nicefrac12}\,,
\end{equation}
so our spaces coincide with the usual Sobolev spaces. We define the dual norms
\begin{align*}
\|f\|_{\underline{H}^{-s}(\cu_n)} & \coloneqq \sup \biggl\{ \fint_{\cu_n} fg : g \in C_c^\infty(\cu_0) \, \|g\|_{\underline{H}^s(\cu_n)} \leq 1 \biggr\}\,, \\
\|f\|_{\Hminusuls{-s}(\cu_n)} & \coloneqq \sup \biggl\{ \fint_{\cu_n} fg : g \in C^\infty(\cu_n) \, \|g\|_{\underline{H}^s(\cu_n)} \leq 1 \biggr\}\,.
\end{align*}
and wherever we use the dual norm we will bound it by
\begin{equation}
\label{e.Hring.def}
\|f\|_{\underline{H}^{-s}(\cu_n)} \leq \|f\|_{\Hminusuls{-s}(\cu_n)} \leq C(d)\|f\|_{\Hring^{-s}} : = C(d)\biggl(\sum_{k=-\infty}^n 3^{2sk} \avsum_{z\in 3^{k-1}\mathbb{Z}^d, z+\cu_k \subseteq \cu_n} |(f)_{z+\cu_k}|^2 \biggr)^{\frac{1}{2}}\,,
\end{equation}
a proof of which can be found in~\cite[Lemma A.1]{AK.HC}. If~$s< \nicefrac12$ then compactly supported functions are dense in~$H^s(\cu_n)$ so~$\|f\|_{\underline{H}^{-s}(\cu_n)} = \|f\|_{\Hminusuls{-s}(\cu_n)}$, while for~$s\geq \nicefrac12$ the norms may be different. We use these spaces to obtain energy estimates for solutions which depend only on the coarse-grained ellipticity constants.

\begin{proposition}
\label{p.gen.poincare}
Suppose that~$\nabla \cdot \a\nabla w =\nabla \cdot \mathbf{f}$ in~$\cu_0$. For every~$s \in (0,1)$ and~$\epsilon \in (0,s)$ there exists a constant~$C = C(s,\epsilon,d)$ such that
\begin{equation}
\label{e.cg.grad}
\|\nabla w\|_{\hat{H}^{-s}(\cu_0)} \leq C\lambda_{s-\epsilon}^{-\nicefrac12}(\cu_0) \|\s^{\nicefrac12}\nabla w\|_{\L^2(\cu_0)} + C\lambda_{s-\epsilon}^{-1}(\cu_0)\|\mathbf{f}\|_{\underline{H}^{s}(\cu_0)}\,,
\end{equation}
and
\begin{equation}
\label{e.cg.flux}
\|\a\nabla w \|_{\hat{H}^{-s}(\cu_0)} \leq C\Lambda_{s}^{\nicefrac12}(\cu_0) \|\s^{\nicefrac12}\nabla w\|_{\L^2(\cu_0)}  + C \frac{\Lambda^{\nicefrac12}_{s}(\cu_0)}{\lambda_{s-\epsilon}^{-\nicefrac12}(\cu_0)} \|\mathbf{f}\|_{\underline{H}^{s}(\cu_0)}\,.
\end{equation}
\end{proposition}

\begin{proof}
\emph{Step 1:}
By~\cite[Lemma A1]{AK.HC}
\begin{align}
\label{e.negative.norm.is}
\|\nabla w\|_{\Hminusuls{-s}(\cu_0)} \leq C \biggl(\sum_{k=-\infty}^0 3^{2sk} \avsum_{z\in 3^{k}\mathbb{Z}^d \cap \cu_0} |(\nabla w)_{z+\cu_k}|^2 \biggr)^{\nicefrac12}\,.
\end{align}
For each~$z\in 3^{k}\mathbb{Z}^d \cap \cu_0$ we let~$\rho_{z}$ be the solution to the Dirichlet problem~$-\nabla \cdot \a\nabla \rho_z = \nabla \cdot (\mathbf{f} - (f)_{z+\cu_k})$ in~$z+\cu_k$ with zero boundary data. From testing the equation for~$\rho_z$ with itself we have
\begin{align*}
\|\s^{\nicefrac12}\nabla \rho_z\|_{\L^2(z+\cu_k)} & \leq [\mathbf{f}]_{\underline{H}^s(z+\cu_k)}^{\nicefrac12}\|\nabla \rho_z\|_{\Hminusuls{-s}(z+\cu_k)}^{\nicefrac12} \\
& \leq C[\mathbf{f}]_{\underline{H}^s(z+\cu_k)}^{\nicefrac12}\|\nabla w\|_{\Hminusuls{-s}(z+\cu_k)}^{\nicefrac12} \\
& \qquad +  C[\mathbf{f}]_{\underline{H}^s(z+\cu_k)}^{\nicefrac12}\|\nabla w-\nabla\rho_z\|_{\Hminusuls{-s}(z+\cu_k)}^{\nicefrac12}\,.
\end{align*}
Now by Proposition~\ref{p.coarse.poincare}
\begin{align*}
\|\nabla w-\nabla\rho_z\|_{\Hminusuls{-s}(z+\cu_k)} \leq C 3^{sk}\lambda_s^{-\nicefrac12}(z+\cu_k)\|\s^{\nicefrac12}(\nabla w-\nabla \rho_z)\|_{\L^2(z+\cu_k)}
\end{align*}
so that combining the above two equations
\begin{align*}
\lefteqn{
\|\s^{\nicefrac12}(\nabla w -\nabla\rho_{z})\|_{\L^2(z+\cu_k)}
} \qquad \qquad & \\
& \leq
\|\s^{\nicefrac12}\nabla w\|_{\L^2(z+\cu_k)} + \|\s^{\nicefrac12}\nabla \rho_z\|_{\L^2(z+\cu_k)}
\\ & \leq
C\|\s^{\nicefrac12}\nabla w\|_{\L^2(z+\cu_k)} + C[\mathbf{f}]_{\underline{H}^s(z+\cu_k)}^{\nicefrac12}\|\nabla w\|_{\Hminusuls{-s}(z+\cu_k)}^{\nicefrac12}
\\ & \qquad
+ C[\mathbf{f}]_{\underline{H}^s(z+\cu_k)}^{\nicefrac12} 3^{\frac{sk}{2}}\lambda_s^{-\nicefrac{1}{4}}(z+\cu_k) \|\s^{\nicefrac12}\nabla (w-\rho_z)\|_{\L^2(z+\cu_k)}^{\nicefrac12}\,.
\end{align*}
Using Young's inequality to re-absorb the energy factor in the last term we obtain
\begin{align}
\label{e.energy.error}
\|\s^{\nicefrac12}\nabla (w -\rho_z)\|_{\L^2(z+\cu_k)} & \leq C\|\s^{\nicefrac12}\nabla w\|_{\L^2(z+\cu_k)} + C[\mathbf{f}]_{\underline{H}^s(z+\cu_k )}^{\nicefrac12}\|\nabla w\|_{\Hminusuls{-s}(z+\cu_k)}^{\nicefrac12} \nonumber \\
& \qquad + C[\mathbf{f}]_{\underline{H}^s(z+\cu_k)} 3^{sk}\lambda_s^{-\nicefrac{1}{2}}(z+\cu_k)
\end{align}
Then using the definition of the dual norm, Proposition~\ref{p.coarse.poincare},~\eqref{e.energy.error}, and Young's inequality, for any~$\delta_k > 0$
\begin{align*}
\lefteqn{
3^{sk}|(\nabla w - \nabla \rho_z)_{z+\cu_k}| 
} \qquad \qquad  & \\ 
& \leq \|\nabla w - \nabla \rho_z\|_{\Hminusuls{-s}(z+\cu_k)} \\
& \leq  C 3^{sk}\lambda_s^{-\nicefrac12}(z+\cu_k) \|\s^{\nicefrac12}\nabla (w-\rho_z)\|_{\L^2(z+\cu_k)} \\
& \leq C 3^{sk}\lambda_s^{-\nicefrac12}(z+\cu_k)\|\s^{\nicefrac12}\nabla w\|_{\L^2(z+\cu_k)} + C\delta_k^{-\nicefrac12}3^{2sk}\lambda_s^{-1}(z+\cu_k)[\mathbf{f}]_{\underline{H}^s(z+\cu_k)} \\
& \qquad + \delta_k^{\nicefrac12} \|\nabla w\|_{\Hminusuls{-s}(z+\cu_k)} \,.
\end{align*}
We are now ready to bound the quantity in~\eqref{e.negative.norm.is}, since we have constructed~$\rho(z+\cu_k)$ to have zero boundary data so that~$(\nabla \rho(z+\cu_k) )_{z+\cu_k} = 0$. Then
\begin{align}
\label{e.start.prop23}
\lefteqn{
\sum_{k=-\infty}^0 3^{2sk} \avsum_{z\in 3^k\mathbb{Z}^d \cap \cu_0}|(\nabla w)_{z+\cu_k}|^2
} \qquad \qquad & \notag \\
& = \sum_{k=-\infty}^0 3^{2sk} \avsum_{z\in 3^k\mathbb{Z}^d \cap \cu_0}|(\nabla w - \nabla\rho(z+\cu_k) ))_{z+\cu_k}|^2 \notag \\
& \leq  C\|\s^{\nicefrac12}\nabla w\|_{\L^2(\cu_0)}^2 \sum_{k=-\infty}^0 3^{2sk}\sup_{z \in 3^k\mathbb{Z}^d \cap \cu_0} \lambda_s^{-1}(z+\cu_k) \notag \\
& \qquad + C \sum_{k=-\infty}^0 3^{4sk}\delta^{-1}_k\sup_{z \in 3^k\mathbb{Z}^d \cap \cu_0} \lambda_s^{-2}(z+\cu_k) \avsum_{z \in 3^k\mathbb{Z}^d \cap \cu_0} [\mathbf{f}]_{\underline{H}^s(z+\cu_k)}^2 \notag \\
& \qquad +  \sum_{k=-\infty}^0 \avsum_{z \in 3^k\mathbb{Z}^d \cap \cu_0} \delta_k \|\nabla w\|^2_{\Hminusuls{-s}(z+\cu_k)}\,.
\end{align}
Now choosing~$\epsilon >0$ and~$\delta_k = \delta_0 3^{2\epsilon k}$ we estimate the ellipticity factor by the inequality~$\sum a_k^2 \leq (\sum a_k )^2$ and
\begin{align}
\label{e.first23}
\sum_{k=-\infty}^0 3^{2(s-\nicefrac{\epsilon}{2})k}\sup_{z\in 3^k\mathbb{Z}^d \cap \cu_0} \lambda_s^{-1}(z+\cu_k) & \leq \sum_{k=-\infty}^1 3^{2(s-\nicefrac{\epsilon}{2})k} \sum_{j=-\infty}^k 3^{2(s-\epsilon)(j-k)} \sup_{z'\in 3^j\mathbb{Z}^d \cap \cu_0} |\s_*^{-1}(z'+\cu_j)| \notag\\
& \leq \frac{C}{\epsilon} \lambda_{s-\epsilon}^{-1}(\cu_0)\,.
\end{align}
The semi-norm of~$f$ is estimated using the integral expression in~\eqref{e.seminorm.integral} as
\begin{align}
\label{e.second23}
\avsum_{z\in 3^k\mathbb{Z}^d \cap \cu_0} [\mathbf{f}]_{\underline{H}^s(z+\cu_k\cap \cu_0)}^2 & \leq C \avsum_{z\in 3^k\mathbb{Z}^d \cap \cu_0} \fint_{z+\cu_k \cap \cu_0} \int_{z+\cu_k \cap \cu_0} \frac{|f(x)-f(y)|^2}{|x-y|^{d+2(s+\epsilon)}} dx dy \notag \\
& \leq C \fint_{\cu_0} \int_{\cu_0} \frac{|f(x)-f(y)|^2}{|x-y|^{d+2(s+\epsilon)}} dx dy\,.
\end{align}
We may re-absorb the last term in~\eqref{e.start.prop23} by taking~$\delta_0$ small enough depending only on~$d, \epsilon$ and~$s$ because
\begin{align}
\label{e.third23}
\sum_{k=-\infty}^0 3^{2\epsilon k} \avsum_{z\in 3^k\mathbb{Z}^d \cap \cu_0} \|\nabla w\|^2_{\Hminusuls{-s}(z+\cu_k)} & \leq \sum_{k=-\infty}^1 3^{2\epsilon k} \sum_{j=-\infty}^k 3^{2sj} \avsum_{\substack{z'\in 3^j\mathbb{Z}^d \cap \cu_0 \\ z\in 3^k\mathbb{Z}^d \cap \cu_0}} |(\nabla w)_{z'+\cu_j}|^2 \notag \\
& \leq \frac{C}{\epsilon} \sum_{j=-\infty}^0 3^{2sj} \avsum_{z'\in 3^j\mathbb{Z}^d \cap \cu_0} |(\nabla w)_{z'+\cu_j}|^2\,.
\end{align}
Inserting~\eqref{e.first23},~\eqref{e.second23} and~\eqref{e.third23} into~\eqref{e.start.prop23} completes the proof of~\eqref{e.cg.grad}.

\noindent
\emph{Step 2: Fluxes}

For the fluxes, if~$z \in 3^k\mathbb{Z}^d \cap \cu_0$ we let~$\chi_{z}$ be the unique solution to
\begin{equation*}
\left\{
\begin{aligned}
& \nabla \cdot \a\nabla \chi_{z} = \nabla \cdot (\mathbf{f} - (\mathbf{f})_{z+\cu_k} )  & \quad \mbox{in} & \quad z+\cu_k \\
& \hat{n}\cdot (\a\nabla\chi_{z}  +\mathbf{f} - (\mathbf{f})_{z+\cu_k}) = 0 & \quad \mbox{on} & \quad  \partial (z + \cu_k )\\
\end{aligned}
\right.
\end{equation*}
Then
\begin{equation*}
\int_{z+\cu_k} \a\nabla \chi_{z} = 0\,.
\end{equation*}
Now testing the equation for~$\chi_{z}$ with itself and then applying~\eqref{e.cg.grad},
\begin{align*}
\|\s^{\nicefrac12}\nabla \chi_{z}\|_{\L^2(z+\cu_k )} & \leq \|\mathbf{f}\|_{\underline{H}^s(z+\cu_k)} \|\nabla \chi_{z}\|_{\underline{\hat{H}}^{-s}(z+\cu_k)} \\
& \leq C 3^{sk}\lambda_{s-\epsilon}^{-\nicefrac12}(z+\cu_k) \|\s^{\nicefrac12}\nabla \chi_{z}\|_{\L^2(z+\cu_k)}[\mathbf{f}]_{\underline{H}^s(z+\cu_k )}  \\
& \qquad + C3^{2sk} \lambda_{s-\epsilon}^{-1}(z+\cu_k) \|\mathbf{f}\|_{\underline{H}^{s}(z+\cu_k)}^2\,,
\end{align*}
so that after re-absorbing the energy factor,
\begin{align*}
\|\s^{\nicefrac12}\nabla \chi_{z}\|_{\L^2(z+\cu_k)} \leq C3^{sk}\lambda_{s-\epsilon}^{-\nicefrac12}(z+\cu_k) \|\mathbf{f}\|_{\underline{H}^{s}(z+\cu_k)}\,.
\end{align*}
Combining~\eqref{e.cg.ineq} with the above display,
\begin{align*}
3^{sk}| (\a\nabla w - \a\nabla \chi_{z})_{z+\cu_k} | & \leq 3^{sk}|\b^{\nicefrac12}(z+\cu_k)| \|\s^{\nicefrac12}\nabla ( w - \chi_{z} ) \|_{\L^2(z+\cu_k)} \\
& \leq 3^{sk}|\b^{\nicefrac12}(z+\cu_k)| \|\s^{\nicefrac12}\nabla  w \|_{\L^2(z+\cu_k)} \\
& \qquad + C3^{2sk}|\b^{\nicefrac12}(z+\cu_k)| \lambda_{s-\epsilon}^{-\nicefrac12}(z+\cu_k) [\mathbf{f}]_{\underline{H}^{s}(z+\cu_k)}\,.
\end{align*}
Using~\cite[Lemma A.1]{AK.HC},~$(\a\nabla w - \a\nabla \chi_{z})_{z+\cu_k} = (\a\nabla w)_{z+\cu_k}$ and the above display,
\begin{align*}
\lefteqn{
\| \a\nabla w\|_{\underline{\hat{H}}^{-s}(\cu_0)}^2
} \quad & \\
& \leq C \sum_{k=-\infty}^0 \avsum_{z\in 3^{k}\mathbb{Z}^d \cap \cu_0} 3^{2sk}| (\a\nabla w - \a\nabla \chi_{z})_{z+\cu_k} |^2  \\
& \leq C\|\s^{\nicefrac12}\nabla w\|_{\L^2(\cu_0)}^2 \Lambda_{s}(\cu_0) \\
& \qquad  +  C \|\mathbf{f}\|_{\underline{H}^{s}(\cu_0)}^2 \sum_{k=-\infty}^0 3^{2sk}\sup_{z\in 3^{k}\mathbb{Z}^d}|\b(z+\cu_k)| 3^{2s k}\sup_{z\in 3^{k}\mathbb{Z}^d \cap \cu_0} \lambda_{s-\epsilon}(z+\cu_k) \\
& \leq C\|\s^{\nicefrac12}\nabla w\|_{\L^2(\cu_0)}^2 \Lambda_{s}(\cu_0) + C \|\mathbf{f}\|_{\underline{H}^{s}(\cu_0)}^2 \Lambda_{s}(\cu_0) \lambda_{s-\epsilon}^{-1}(\cu_0)\,.
\end{align*}
\end{proof}

Our next proposition is a coarse-grained energy estimate for solutions.

\begin{proposition}
\label{p.boundary.data}
For every~$s\in (0,1)$ and~$\epsilon \in (0,s)$ there exists a constant~$C= C(d,s,\epsilon)$ such that if~$u\in H^1_{\s}(\cu_0)$ satisfies
\begin{align}
\label{e.PDE.problem}
\left\{
\begin{aligned}
& -\nabla \cdot \a \nabla u = \nabla \cdot \mathbf{f} & \mbox{in } \cu_0\,, \\
& u= h & \mbox{on } \partial \cu_0\,,
\end{aligned}
\right.
\end{align}
then
\begin{align}
\label{e.Dirichlet.energy}
\|\s^{\nicefrac12} \nabla u\|_{\L^2(\cu_0)}  \leq C \Lambda_{s}^{\nicefrac12}(\cu_0) \|\nabla h\|_{\underline{H}^s(\cu_0)} + C \lambda_{s-\epsilon}^{-\nicefrac12}(\cu_0) \|\mathbf{f}\|_{\underline{H}^{s}(\cu_0)}\,,
\end{align}
and if instead~$u\in H^1_{\s}(\cu_0)$ solves the Neumann problem
\begin{align}
\label{e.PDE.Neumann}
\left\{
\begin{aligned}
& \nabla \cdot \a \nabla u = \nabla \cdot \mathbf{f} & \mbox{in } \cu_0\,, \\
& \hat{n} \cdot (\a\nabla u - \mathbf{f}) = 0 & \mbox{on } \partial \cu_0\,,
\end{aligned}
\right.
\end{align}
then
\begin{align}
\label{e.Neumann.energy}
\|\s^{\nicefrac12} \nabla u\|_{\L^2(\cu_0)}  \leq C \lambda_{s-\epsilon}^{-\nicefrac12}(\cu_0) \|\mathbf{f}\|_{\underline{H}^{s}(\cu_0)}\,.
\end{align}
\end{proposition}

\begin{proof}
We first consider the Dirichlet problem. By linearity we can write~$u = v+w$ where~$v$ solves~\eqref{e.PDE.problem} with~$\mathbf{f}=0$ and~$w$ solves~\eqref{e.PDE.problem} with~$h=0$. We then estimate the energy of~$u$ by the triangle inequality.

\smallskip
\emph{Step 1:}
We first consider~$v$ and test the equation for~$v$ with~$v-h$ (justified as in Section~\ref{s.coarse.graining}) to obtain
\begin{align}
\label{e.test.energy}
\fint_{\cu_0} \nabla v \cdot \s\nabla v & = \fint_{\cu_0} \a\nabla v \cdot \nabla h \leq \|\a\nabla v\|_{\Hminusuls{-s}(\cu_0)}\|\nabla h\|_{\underline{H}^s(\cu_0)} \leq C\Lambda_s^{\nicefrac12}(\cu_0)\|\s^{\nicefrac12}\nabla v\|_{\L^2(\cu_0)} \|\nabla h\|_{\underline{H}^s(\cu_0)}\,,
\end{align}
using Proposition~\ref{p.coarse.poincare}. Putting these bounds into~\eqref{e.test.energy} and re-absorbing the energy term we conclude that
\begin{equation}
\label{e.no.RHS}
\|\s^{\nicefrac12}\nabla v\|_{\L^2(\cu_0)} \leq C\Lambda_{s}^{\nicefrac12}(\cu_0) \|\nabla h\|_{\underline{H}^s(\cu_0)}\,.
\end{equation}

\smallskip
\emph{Step 2:}
We now consider~$w$ and test the equation for~$w$ with itself to obtain
\begin{align*}
\fint_{\cu_0} \nabla w \cdot \s\nabla w & = -\fint_{\cu_0} \nabla w \cdot \mathbf{f} \leq \|\mathbf{f}\|_{\underline{H}^s(\cu_0)}\|\nabla w\|_{\Hminusuls{-s}(\cu_0)} \\
& \leq C\lambda_{s-\epsilon}^{-1}(\cu_0)\|\mathbf{f}\|^2_{\underline{H}^{s}(\cu_0)} + C\lambda_{s-\epsilon}^{-\nicefrac12}(\cu_0) \|\mathbf{f}\|_{\underline{H}^{s}(\cu_0)}\|\s^{\nicefrac12}\nabla w\|_{\L^2(\cu_0)}\,.
\end{align*}
Using Young's inequality again to absorb the energy term on the right-hand side we conclude our proof for the Dirichlet problem.

\smallskip
\emph{Step 3:}
In the case of the Neumann problem~\eqref{e.PDE.Neumann} the proof is nearly identical. Testing the equation with the solution we obtain
\begin{align*}
\fint_{\cu_0} \nabla u \cdot \s \nabla u & = \fint_{\cu_0} \mathbf{f} \cdot \nabla u\leq \|\mathbf{f}\|_{\underline{H}^{s}(\cu_0)} \|\nabla u\|_{\Hminusuls{-s}(\cu_0)} \\
& \leq C\lambda_{s-\epsilon}^{-1}(\cu_0)\|\mathbf{f}\|^2_{\underline{H}^{s}(\cu_0)} + C\lambda_{s-\epsilon}^{-\nicefrac12}(\cu_0) \|\mathbf{f}\|_{\underline{H}^{s}(\cu_0)}\|\s^{\nicefrac12}\nabla u\|_{\L^2(\cu_0)}\,.
\end{align*}
We then conclude by using Young's inequality to re-absorb the energy term on the right-hand side.
\end{proof}

The energy estimates in Proposition~\ref{p.boundary.data} allow us to solve the PDE~\eqref{e.PDE.problem} with data~$\mathbf{f}\in H^s(\cu_0)$ and~$h\in H^{1+s}(\cu_0)$ provided that the coarse-grained ellipticity constants~$\Lambda_s(\cu_0)$ and~$\lambda_s^{-1}(\cu_0)$ are finite (with~$s$ replaced by~$s-\epsilon$ for any~$\epsilon >0$). We briefly sketch the argument here. By linearity consider the cases~$h=0$ and~$\mathbf{f}=0$ separately. If~$\mathbf{f}\in H^s(\cu_0)$ we take smooth approximations~$\mathbf{f}_j \to \mathbf{f}$ in~$H^s(\cu_0)$ and let~$u_j$ solve~$-\nabla \cdot \a\nabla u_j = \nabla \cdot \mathbf{f}_j$ with zero Dirichlet data. By~\eqref{e.Dirichlet.energy} we have
\begin{equation*}
\|\s^{\nicefrac12}\nabla (u_j - u_k)\|_{\L^2(\cu_0)} \leq C\lambda_{s-\epsilon}^{-\nicefrac12}(\cu_0)\|\mathbf{f}_j - \mathbf{f}_k\|_{\underline{H}^s(\cu_0)}\,,
\end{equation*}
and therefore~$u_j$ is a Cauchy sequence in~$H^1_{s,0}(\cu_0)$. The limit~$u \in H^1_{s,0}(\cu_0)$ then satisfies~$-\nabla \cdot \a\nabla u = \nabla \cdot \mathbf{f}$ by passing to the limit in the weak formulation of the equation for~$u_j$. Similarly, if~$h_j \to h$ in~$H^{1+s}(\cu_0)$ and~$v_j$ solves~$-\nabla \cdot \a\nabla v_j = 0$ with boundary data~$h_j$ then
\begin{equation*}
\|\s^{\nicefrac12} \nabla (v_j-v_k)\|_{\L^2(\cu_0)}  \leq C\bigl( |\b(\cu_0)|^{\nicefrac12} + \Lambda_{s}^{\nicefrac12}(\cu_0)\bigr) \|\nabla (h_j-h_k)\|_{\underline{H}^s(\cu_0)}\,.
\end{equation*}
The sequence~$v_j$ then has a limit~$v\in H^1_{\s}(\cu_0)$, which solves~$-\nabla \cdot \a\nabla v = 0$ by passing to the limit in the weak formulation of the problem for~$v_j$. To check that we have the correct boundary data we compute
\begin{align*}
\|\mathrm{Tr}(v-h)\|_{L^1(\partial \cu_0)} & \leq \lim_{j\to\infty} \bigl( \|\mathrm{Tr}(v_j-v)\|_{L^1(\partial \cu_0)} + \|\mathrm{Tr}(h_j-h)\|_{L^1(\partial \cu_0)} \bigr)\\
& \leq C\lim_{j\to\infty}\bigl( \|\s^{\nicefrac12}\nabla (v_j-v)\|_{\L^2(\cu_0)} + \|\nabla (h_j - h)\|_{\L^2(\cu_0)} \bigr) = 0\,.
\end{align*}

\section{Convergence of the coarse-grained matrices}
\label{s.ellip.converge}

In Section~\ref{s.coarse.graining} we showed that we can define the coarse-grained coefficients and make sense of the Dirichlet problem under our assumptions on the coefficient fields. In Section~\ref{ss.ergodic} we apply the ergodic theorem to obtain the homogenized matrices~$\shom,\shom_*$ and~$\khom$. In Section~\ref{ss.get.ahom} we show that the homogenization error~$\shom - \shom_*$ vanishes.

\subsection{Ergodic theorem}
\label{ss.ergodic}

In this section we apply the ergodic theorem to the coarse-grained coefficients to define our candidate for the homogenized matrix~$\ahom$. By~\eqref{e.in.L1}, stationarity, and the ellipticity bounds in~\eqref{e.cg.mat.bound}, the expectations of the coarse-grained matrices are finite. We can then define 
\begin{align}
\label{e.def.homs}
\left\{
\begin{aligned}
& \shom_*(U) \coloneqq \mathbb{E}[\s_*^{-1}(U)]^{-1}\,, \\
& \khom (U) \coloneqq \shom_*(U)\mathbb{E}[\s_*^{-1}(U)\k(U)]\,, \\
& \shom(U) \coloneqq \mathbb{E}[\s(U) + \k^t(U)\s_*^{-1}(U)\k(U)] - \khom^t(U)\shom_*^{-1}(U)\khom(U)\,, \\
\end{aligned}
\right.
\end{align}
and we let
\begin{equation}
\bhom(U) \coloneqq \shom(U) + \khom^t(U) \shom_*^{-1}(U)\khom(U)\,.
\end{equation}
Although we put bars over the top of the matrices to denote the homogenized quantities, note that only the subadditive quantities~$\shom_*^{-1}(U)$ and~$\bhom(U)$ are actually the expectation of the corresponding coarse-grained matrices. The reason for our choice is that we can write
\begin{equation}
\label{e.bfAhom.def}
\bfAhom(U) \coloneqq \begin{pmatrix}
\shom(U) + \khom^t(U)\shom_*^{-1}(U)\khom(U) & -\khom^t(U)\shom_*^{-1}(U) \\ -\shom_*^{-1}(U)\khom(U) & \shom_*^{-1}(U)
\end{pmatrix}
= \mathbb{E}[\bfA(U)]\,.
\end{equation}

The following proposition is a slight modification of the Kingman subadditive ergodic theorem (for instance see~\cite[Proposition 1.11]{AK.Book}) to include a non-uniform upper bound.

\begin{proposition}
\label{p.ergodic.thm}
Suppose that~$\mathbb{P}$ is~$\mathbb{Z}^d$-stationary and ergodic. Let~$\mu$ be a stationary function that is subadditive: if~$\{U_i\}$ are a partition of~$U$ then
\begin{equation*}
\mu(U) \leq \sum_{i} \frac{|U_i|}{|U|}\mu(U_i)\,.
\end{equation*}
Suppose also that there exists a random field~$g:\mathbb{R}^d\to\mathbb{R}$ with~$\mathbb{E}[\fint_{\cu_0} g] < \infty$ such that for every bounded domain~$U\subset \mathbb{R}^d$
\begin{equation*}
0 \leq \mu(U) \leq \fint_U g\,.
\end{equation*}
Then there exists a constant~$\bar{\mu}$ such that for every bounded Lipschitz domain~$U\subseteq \mathbb{R}^d$
\begin{equation*}
\mathbb{P}\biggl[\lim_{t\to\infty}\mu(tU) = \bar{\mu} \biggr] = 1\,.
\end{equation*}
Moreover~$\bar{\mu} = \lim_{n\to\infty}\mathbb{E}[\mu(\cu_n)]$.
\end{proposition}

\begin{proof}
\emph{Step 1: Identification of~$\bar{\mu}$.}
For any~$n\in\mathbb{N}$ the random field
\begin{equation*}
f_n \coloneqq \sum_{z\in 3^n\mathbb{Z}^d} \mu(z+\cu_n)\indc_{z+\cu_n}
\end{equation*}
is stationary with respect to~$3^n\mathbb{Z}^d$ translations with~$\mathbb{E}[ \fint_U f_n ] = \mathbb{E}[\mu(\cu_n)]$ for any domain~$U$. By the multi-parameter Wiener ergodic theorem (see~\cite[Proposition 1.7]{AK.Book}) for any domain~$U$
\begin{equation*}
\limsup_{t\to\infty} \biggl| \fint_{tU} f_n - \mathbb{E}[\mu(\cu_n)] \biggr| = 0 \quad \mbox{almost surely}\,.
\end{equation*}
By subadditivity and stationarity the sequence~$\mathbb{E}[\mu(\cu_n)]$ is monotone decreasing and therefore there exists a limit~$\bar{\mu}$.

\emph{Step 2: Convergence.}
Given any bounded domain~$U$ and~$n\in\mathbb{N}$ we let
\begin{equation*}
V_n(U) : = \bigl\{ z+\cu_n: z\in 3^n\mathbb{Z}^d\,, z+\cu_n \subseteq U\bigr\}
\end{equation*}
be the subset of cubes of size~$\cu_n$ contained in~$U$. Then by sub-additivity
\begin{align*}
\mu(tU) & \leq \frac{|V_n(tU)|}{|tU|}\mu(V_n(tU)) + \frac{|tU\setminus V_n(tU)|}{|tU|}\mu(tU\setminus V_n(tU))\,.
\end{align*}
The first term we bound by sub-additivity as
\begin{align*}
\frac{|V_n(tU)|}{|tU|}\mu(V_n) & \leq \frac{|V_n(tU)|}{|tU|} \avsum_{z\in V_n(tU)}\mu(z+\cu_n) = \frac{|V_n(tU)|}{|tU|} \fint_{V_n(tU)} f_n \leq \fint_{tU} f_n\,,
\end{align*}
while the second is bounded using
\begin{align*}
\frac{|tU\setminus V_n(tU)|}{|tU|}\mu(tU\setminus V_n(tU)) \leq \frac{|tU\setminus V_n(tU)|}{|tU|} \fint_{tU \setminus V_n(tU)} g = \frac{1}{|tU|}\int_{tU \setminus V_n(tU)} g\,.
\end{align*}
If~$\epsilon > 0$ is given,~$n\in\mathbb{N}$ is fixed and~$U_\epsilon \coloneqq \{ x\in U : \mathrm{dist}(x,\partial U) < \epsilon\}$ then there exists~$T(\epsilon,n)$ large enough that when~$t\geq T$ we have~$tU\setminus V_n(tU) \subseteq t U_\epsilon$. Then
\begin{equation*}
\frac{1}{|tU|}\int_{tU \setminus V_n(tU)} g\leq \frac{1}{|tU|}\int_{tU_{\epsilon}} g = \frac{|tU_{\epsilon}|}{|tU|}\fint_{tU_{\epsilon}} g = \frac{|U_{\epsilon}|}{|U|}\fint_{tU_{\epsilon}} g\,.
\end{equation*}
Putting our estimates together we obtain that almost surely
\begin{equation*}
\limsup_{t\to\infty}\mu(tU) \leq \limsup_{t\to\infty}\fint_{tU}f_n + \frac{|U_{\epsilon}|}{|U|}\limsup_{t\to\infty} \fint_{tU_{\epsilon}} g = \mathbb{E}[\mu(\cu_n)] + \frac{|U_{\epsilon}|}{|U|}\mathbb{E}\biggl[ \fint_{\cu_0} g \biggr]\,,
\end{equation*}
since~$g$ satisfies the multi-parameter Wiener ergodic theorem. Taking~$\epsilon \to 0$ and~$n\to\infty$ we conclude that
\begin{equation*}
\limsup_{t\to\infty}\mu(tU) \leq \bar{\mu}\,.
\end{equation*}
The lower bound~$\liminf_{t\to\infty} \mu(tU) \geq \bar{\mu}$ is exactly as in~\cite[Proposition 1.11]{AK.Book}.
\end{proof}

We now use finiteness of the first moments in~\ref{a.finite.E} to apply the ergodic theorem to the coarse-grained matrices. At this point we do not yet need the coarse-grained ellipticity assumption~\ref{a.ellipticity}.

\begin{proposition}
\label{p.Ahom}
There exists a deterministic matrix~$\bfAhom$ such that for any bounded Lipschitz domain~$U$
\begin{equation*}
\lim_{t\to\infty}\bfA(tU) = \bfAhom \quad \mbox{almost surely}\,.
\end{equation*}
and there exist deterministic matrices~$\shom,\shom_*^{-1}$ and~$\khom$ such that~$\bfAhom$ has the form
\begin{equation}
\label{e.ergodic}
\bfAhom = \begin{pmatrix}
\shom + \khom^t\shom_*^{-1}\khom & -\khom^t\shom_*^{-1} \\
-\shom_*^{-1}\khom & \shom_*^{-1}
\end{pmatrix}\,.
\end{equation}
Moreover,~$\bfAhom,\shom$ and~$\shom_*$ are the monotone limits of~$\bfAhom(\cu_n),\shom(\cu_n)$ and~$\shom_*(\cu_n)$ respectively, we have the bounds
\begin{equation}
\label{e.ellip.hom}
\mathbb{E}\biggl[ \fint_{\cu_0} \s^{-1} \biggr]^{-1} \leq \shom_* \leq \shom \leq \shom+\khom^t\shom_*^{-1}\khom \leq \mathbb{E} \biggl[\fint_{\cu_0} \s + \k^t\s^{-1}\k \biggr]
\end{equation}
and
\begin{equation}
\label{e.sym.khom}
-(\shom - \shom_*) \leq \frac{1}{2}(\khom + \khom^t) \leq \shom - \shom_*\,.
\end{equation}
\end{proposition}
\begin{proof}
The first statement is an immediate corollary of Proposition~\ref{p.ergodic.thm}, which also implies that~$\bfAhom$ is the monotone limit of~$\bfAhom(\cu_n)$. If we label the upper-left block of~$\bfAhom$ as~$\bhom$ and the lower right-block as~$\shom_*^{-1}$ then we immediately obtain the monotone convergence of~$\bhom(\cu_n)$ and~$\shom_*^{-1}(\cu_n)$ to~$\bhom$ and~$\shom_*^{-1}$ respectively. From convergence of the off-diagonal entries, namely that there exists~$\bfAhom_{21}$ such that~$-\shom_*^{-1}(\cu_n)\khom(\cu_n) \to \bfAhom_{21}$, it follows that
\begin{equation*}
\khom(\cu_n) \to \khom \coloneqq -\shom_*\bfAhom_{21}\,,
\end{equation*}
In order to see that~$\shom(\cu_n)$ is monotone we compute for any~$m\in\mathbb{N}$
\begin{align*}
\mathbb{E}[J(\cu_m),p,-\khom(\cu_n)p)] & + \mathbb{E}[J^*(\cu_m,p,\khom(\cu_n)p)] \\
& = p \cdot \shom(\cu_m)p + (\khom(\cu_m)-\khom(\cu_n))p \cdot \shom_*^{-1}(\cu_m) (\khom(\cu_m) - \khom(\cu_n))p\,.
\end{align*}
Now by stationarity and subadditivity, for any~$m\geq n$
\begin{equation*}
\mathbb{E}[J(\cu_m),p,q)]  + \mathbb{E}[J^*(\cu_m,p,q')]  \leq \mathbb{E}[J(\cu_n),p,q)] + \mathbb{E}[J^*(\cu_n,p,q')]
\end{equation*}
so in particular with our choice of~$q = -\khom(\cu_n)p$ and~$q'=\khom(\cu_n)p$ we get
\begin{equation*}
\shom(\cu_m) + (\khom(\cu_m) - \khom(\cu_n))^t \shom_*^{-1}(\cu_m) (\khom(\cu_m)-\khom(\cu_n)) \leq \shom(\cu_n)\,,
\end{equation*}
which implies that~$\shom(\cu_m) \leq \shom(\cu_n)$. It follows that the limit~$\shom$ exists, and
\begin{align*}
\shom = \lim_{n\to\infty}\shom(\cu_n) = \lim_{n\to\infty} \bigl(\bhom(\cu_n) - \khom^t(\cu_n)\shom_*^{-1}(\cu_n)\khom(\cu_n)\bigr) = \bhom - \khom^t\shom_*\khom\,,
\end{align*}
so that~$\bfA$ has the form in~\eqref{e.ergodic}.

To obtain the ordering~$\shom_*(\cu_n) \leq \shom(\cu_n)$, if we take the expectation of~$J(U,p,q)$ in~\eqref{e.Jmat} and~$J^*(U,p,q')$ in~\eqref{e.Jstar.mat} with any~$p\in\mathbb{R}^d$ and the choices~$q = (\shom(\cu_n) - \khom(\cu_n))p$ and~$q' = (\shom(\cu_n) - \khom(\cu_n))p$ then 
\begin{align*}
0 \leq \mathbb{E}[J(\cu_n,p,(\shom(\cu_n) - \khom(\cu_n))p)] + \mathbb{E}[J^*(\cu_n,p, (\shom(\cu_n) + \khom(\cu_n)p)] = \frac{1}{2}p \cdot \bigl(\shom(\cu_n) - \shom_*(\cu_n) \bigr) p\,.
\end{align*}
Using this ordering and taking expectations of the ellipticity bounds in~\eqref{e.cg.mat.bound} we have the uniform ellipticity bounds in~\eqref{e.ellip.hom}.

Finally,
\begin{align*}
0 \leq \mathbb{E}[J(\cu_n,p,(\shom(\cu_n) - \khom(\cu_n))p)] = \frac{1}{2}p \cdot \bigl(\shom(\cu_n) - \shom_*(\cu_n) \bigr)p + p \cdot \khom(\cu_n) p
\end{align*}
and
\begin{equation*}
0 \leq \mathbb{E}[J^*(\cu_n,p,(\shom(\cu_n) + \khom(\cu_n))p)] = \frac{1}{2}p \cdot \bigl(\shom(\cu_n) - \shom_*(\cu_n) \bigr)p - p \cdot \khom(\cu_n) p
\end{equation*}
and therefore the symmetric part of~$\khom(\cu_n)$ is controlled by the difference between~$\shom(\cu_n)$ and~$\shom_*(\cu_n)$:
\begin{equation}
\label{e.kn.sym}
-(\shom(\cu_n) - \shom_*(\cu_n)) \leq \frac{1}{2}(\khom(\cu_n)+\khom^t(\cu_n)) \leq \shom(\cu_n) - \shom_*(\cu_n)\,.
\end{equation}
We obtain~\eqref{e.sym.khom} by taking a limit.
\end{proof}

Although we now have convergence of the coarse-grained matrices, we require two more steps in order to prove homogenization. First, in~$\shom$ and~$\shom_*$ we have an upper and lower bound for the symmetric part of the homogenized matrix, but we do not yet know that they are equal. Second, to prove homogenization of the Dirichlet problem we need to sum the effect of the coarse-grained matrices over every scale, and we do not know that this effect is finite. It is for this reason that we introduce the coarse-grained ellipticity constants in~\eqref{e.Lambda.def} and~\eqref{e.lambda.def}. We will tackle the first problem in Section~\ref{ss.get.ahom} and the second in Section~\ref{s.homogenize}.

\subsection{Convergence of the homogenization error}
\label{ss.get.ahom}

In Section~\ref{ss.ergodic} we showed that the coarse-grained matrices~$\s(\cu_n),\s_*^{-1}(\cu_n)$ and~$\k(\cu_n)$ converge by the ergodic theorem. Using the stronger ellipticity assumption~\ref{a.ellipticity} we are able to show that the limiting homogenization error~$\shom - \shom_*$ vanishes. The goal of this section is to prove the following proposition:
\begin{proposition}
\label{p.ahom}
Using assumption~\ref{a.ellipticity} we have
\begin{equation}
\label{e.ahom}
\shom = \shom_* \quad \mbox{and} \quad \khom = -\khom^t\,.
\end{equation}
\end{proposition}
Note that the second statement in~\eqref{e.ahom} follows from the first by~\eqref{e.sym.khom}. Our first step in the proof of Proposition~\ref{p.ahom} is to relate the problem to the energy of solutions, via the representation~\eqref{e.J.energy} in terms of~$J(U,p,q)$. By re-arranging~\eqref{e.Jmat} and~\eqref{e.Jstar.mat} and applying Proposition~\ref{p.Ahom}, for any~$p,q\in\mathbb{R}^d$,
\begin{align*}
\lim_{t\to\infty} J(tU,p,q) & = \lim_{t\to\infty} \biggl( \frac{1}{2}p \cdot \b(tU)p + \frac{1}{2}q \cdot \s_*^{-1}(tU)q + p \cdot \k^t(tU)\s_*^{-1}(tU)q - p\cdot q \biggr) \\
& = \frac{1}{2} p\cdot \bhom p + \frac{1}{2} q \cdot \shom_*^{-1} q + p \cdot \khom^t\shom_*^{-1}q - p\cdot q\,,
\end{align*}
and similarly
\begin{align*}
\lim_{t\to\infty} J^*(tU,p,q') = \frac{1}{2} p\cdot \bhom p + \frac{1}{2} q' \cdot \shom_*^{-1} q' - p \cdot \khom^t\shom_*^{-1}q' - p\cdot q'\,.
\end{align*}
If we add these together and make the choices
\begin{equation}
\label{qs.choice}
q = (\shom_* - \khom)p \quad \mbox{and} \quad q'=(\shom_* + \khom)p
\end{equation}
then simple algebraic manipulations give
\begin{align}
\label{e.s.gap}
\lim_{t\to\infty} \bigl( J(tU,p,q) + J^*(tU,p,q') \bigr) =  p \cdot \bigl(\shom - \shom_* \bigr) p\,.
\end{align}
Our goal is then to show that the limit on the left-hand side vanishes. We only need to choose one domain, and we will take~$U = \cu_0$ because we can conveniently partition the domain into subcubes. We denote the mean-zero maximizer of~\eqref{e.J.def} by~$v(U) = v(U,p,q)$ with the understanding that~$p$ and~$q$ are fixed, either arbitrarily for general propositions or as in~\eqref{qs.choice}.

\begin{proposition}
\label{p.J.bound}
There exists a constant~$C=C(d)$ such that for any~$p,q\in\mathbb{R}^d$ and~$0<s,t<1$ with~$s+t = 1$,
\begin{align*}
\limsup_{n\to\infty} J(\cu_{n-1},p,q) \leq C \limsup_{n\to\infty} \biggl( 3^{-sn}\|\nabla v(\cu_n)\|_{\underline{H}^{-s}(\cu_n)} 3^{-tn}\|\a \nabla v(\cu_n)\|_{\underline{H}^{-s}(\cu_n)}\biggr)\,.
\end{align*}
\end{proposition}
\begin{proof}
Using~\eqref{e.J.energy} and the inequality~$\frac{1}{2}a^2 \leq b^2 + (a-b)^2$
\begin{align}
\label{e.start.J.bound}
J(\cu_{n-1},p,q) & = \frac{1}{2}\fint_{\cu_{n-1}} \nabla v(\cu_{n-1})\cdot \s \nabla v(\cu_{n-1}) \nonumber \\
& \leq \fint_{\cu_{n-1}} \nabla v(\cu_n) \cdot \s \nabla v(\cu_n) + \fint_{\cu_{n-1}} \bigl(\nabla v (\cu_{n-1}) - \nabla v(\cu_n) \bigr) \cdot \s \bigl(\nabla v (\cu_{n-1}) - \nabla v(\cu_n) \bigr)\,.
\end{align}
For the second term we use the quadratic response in~\eqref{e.quad.response} to estimate
\begin{align*}
\lefteqn{
\fint_{\cu_{n-1}} \bigl(\nabla v (\cu_{n-1}) - \nabla v(\cu_n) \bigr) \cdot \s \bigl(\nabla v (\cu_{n-1}) - \nabla v(\cu_n) \bigr)
} \quad & \\
& \leq \sum_{z\in 3^{n-1}\mathbb{Z}^d \cap \cu_n} \bigl(\nabla v (z+\cu_{n-1}) - \nabla v(\cu_n) \bigr) \cdot \s \bigl(\nabla v (z+\cu_{n-1}) - \nabla v(\cu_n) \bigr) \\
& \leq \sum_{z\in 3^{n-1}\mathbb{Z}^d \cap \cu_n}\biggl( J(z+\cu_{n-1}) - \fint_{z+\cu_{n-1}} \biggl(-\frac{1}{2}\nabla v(\cu_n)\cdot \s\nabla v(\cu_n) - p \cdot \a \nabla v(\cu_n) + q\cdot \nabla v(\cu_n) \biggr) \biggr) \\
& \leq C(d) \avsum_{z\in 3^{n-1}\mathbb{Z}^d \cap \cu_n} \bigl( J(z+\cu_{n-1},p,q) - J(\cu_n,p,q) \bigr)\,.
\end{align*}
By Proposition~\ref{p.Ahom} and the representation of~$J(U,p,q)$ in~\eqref{e.J.bigA}, the above vanishes as~$n\to\infty$. We are therefore left only with estimating the first term in~\eqref{e.start.J.bound}.

We fix a cutoff function~$\varphi \in C_c^\infty(\cu_n)$ which satisfies
\begin{equation*}
\varphi|_{\cu_{n-1}} = 1\,, \quad \|\nabla \varphi\|_{L^\infty(\cu_n)} \leq C3^{-n}\,, \quad \mbox{and} \quad \|\nabla^2\varphi\|_{L^\infty(\cu_n)} \leq C3^{-2n}\,.
\end{equation*}
Testing the equation for~$v(\cu_n)$ with~$\varphi v(\cu_n)$ and using a duality splitting we have, for any~$0\leq s \leq 1$ and~$t=1-s$,
\begin{align*}
\fint_{\cu_{n-1}}\nabla v(\cu_n)\cdot \s \nabla v(\cu_n)& \leq C\fint_{\cu_n} \varphi \nabla v(\cu_n) \cdot \s \nabla v(\cu_n) = -C\fint_{\cu_n} v(\cu_n) \nabla \varphi \cdot \a\nabla v(\cu_n) \\
& \leq C \|v(\cu_n)\nabla \varphi\|_{\underline{H}^t(\cu_n)} \|\a\nabla v\|_{\underline{H}^{-t}(\cu_n)} \\
& \leq C 3^{-n}\|\nabla v(\cu_n)\|_{\underline{H}^{-s}(\cu_n)}\|\a\nabla v\|_{\underline{H}^{-t}(\cu_n)}\,,
\end{align*}
where in the last line we have used Proposition~\ref{p.cutoff.func}.
\end{proof}

Our next two propositions bound the weak norms on the right-hand side of Proposition~\ref{p.J.bound}. For notational convenience we define the random variables
\begin{equation}
\label{e.mathcalG}
\mathcal{G}_{s,l}(\cu_n) \coloneqq |\shom_*-\khom|^2\sum_{k=n-l+1}^n 3^{2s(k-n)}  \biggl| \avsum_{z\in 3^{k-1}\mathbb{Z}^d \cap \cu_n} \bfA(z+\cu_k) - \bfA(\cu_n) \biggr|
\end{equation}
and
\begin{equation}
\label{e.mathcalH}
\mathcal{H}_{s,l}(\cu_n) \coloneqq |\shom_*-\khom|^2\sum_{k=n-l+1}^{n} 3^{2s(k-n)} \avsum_{z\in 3^{k-1}\mathbb{Z}^d \cap \cu_n} \bigl|\bfA(z+\cu_k) - \bfAhom\bigr|^2
\end{equation}
Here~$\mathcal{G}_{s,l}(\cu_n)$ measures the subadditivity defect in~$\bfA(\cu_n)$, while~$\mathcal{H}_{s,l}(\cu_n)$ measures fluctuations.

\begin{proposition}
\label{p.grad.norm}
For any~$\epsilon > 0$ there exists a constant~$C=C(d,\epsilon)$ such that for any~$s\in (0,1),l\in\mathbb{N}$ and~$p\in\mathbb{R}^d$ with~$|p|=1$, if~$q\in\mathbb{R}^d$ is given by~\eqref{qs.choice} then
\begin{align*}
3^{-sn}\|\nabla v(\cu_n)\|_{\Hring^{-s}(\cu_n)} & \leq C\bigl( \lambda_{s-\epsilon}^{-\nicefrac12}(\cu_n) \mathcal{G}_{\epsilon,l}^{\nicefrac12}(\cu_n) + \mathcal{H}_{s,l}^{\nicefrac12}(\cu_n) +  3^{-\epsilon l}\lambda_{s-\epsilon}^{-\nicefrac12}(\cu_n)\|\s^{\nicefrac12}\nabla v(\cu_n)\|_{\L^2(\cu_n)} \bigr)\,.
\end{align*}
\end{proposition}
\begin{proof}
\emph{Step 1: Small scales.}
Fixing~$l\in\mathbb{N}$, we first consider scales~$k\leq n -l$, where we will gain a scale discount factor. Using~\eqref{e.cg.ineq}, the fact that the sum over scales is an~$\ell^q$ norm and H\"older's inequality
\begin{align*}
\lefteqn{
\sum_{k=-\infty}^{n-l} 3^{2s(k-n)} \avsum_{z\in 3^{k-1}\mathbb{Z}^d\cap \cu_n} |(\nabla v(\cu_n))_{z+\cu_k}|^2 
} \qquad \qquad & \\
& \leq \sum_{k=-\infty}^{n-l} 3^{2s(k-n)} \sup_{z\in 3^{k-1}\mathbb{Z}^d\cap \cu_n}|\s_*^{-1}(z+\cu_k)| \avsum_{z\in 3^{k-1}\mathbb{Z}^d\cap \cu_n} |\s_*^{\nicefrac12}(z+\cu_k)(\nabla v(\cu_n))_{z+\cu_k}|^2  \\
& \leq 3^{-2sl}\|\s^{\nicefrac12}\nabla v(\cu_n)\|^2_{\L^2(\cu_n)} \sum_{k=-\infty}^{n-l} 3^{2s(k-n+l)}\sup_{z\in 3^{k-1}\mathbb{Z}^d\cap \cu_n} |\s_*^{-1}(z+\cu_k)| \\
& \leq 3^{-2sl} \|\s^{\nicefrac12}\nabla v(\cu_n)\|^2_{\L^2(\cu_n)} \biggl(\sum_{k=-\infty}^{n-l} 3^{s(k-n+l)}\sup_{z\in 3^{k-1}\mathbb{Z}^d\cap \cu_n} |\s_*^{-1}(z+\cu_k)|^{\nicefrac12}\biggr)^{2} \\
& \leq C 3^{-2sl} \|\s^{\nicefrac12}\nabla v(\cu_n)\|^2_{\L^2(\cu_n)} \epsilon^{-1} \sum_{k=-\infty}^{n-l} 3^{2(s-\epsilon)(k-n+l)}\sup_{z\in 3^{k-1}\mathbb{Z}^d\cap \cu_n} |\s_*^{-1}(z+\cu_k)| \\
& \leq C3^{-2\epsilon l} \epsilon^{-1} \|\s^{\nicefrac12}\nabla v(\cu_n)\|^2_{\L^2(\cu_n)} \lambda_{s-\epsilon}^{-1}(\cu_n)\,.
\end{align*}
\emph{Step 2: Mesoscopic scales.}
At every scale~$k > n-l$ we have by the triangle inequality
\begin{align}
\label{e.triangle}
|(\nabla v(\cu_n))_{z+\cu_k}| & \leq |(\nabla v(\cu_n) - \nabla v(z+\cu_k))_{z+\cu_k}| + |(\nabla v(z+\cu_k))_{z+\cu_k}|
\end{align}
For the first term we use the coarse-graining inequality~\eqref{e.cg.ineq} and then the quadratic response of~$J$ in~\eqref{e.quad.response} to obtain
\begin{align}
\label{e.grad.av.step2}
\lefteqn{
\avsum_{z\in 3^{k-1}\mathbb{Z}^d \cap \cu_n}  |(\nabla v(\cu_n) - \nabla v(z+\cu_k))_{z+\cu_k}|^2
} \qquad \qquad &\nonumber \\
& \leq \avsum_{z\in 3^{k-1}\mathbb{Z}^d \cap \cu_n} |\s_*^{-1}(z+\cu_k)| |\s_*^{\nicefrac12}(z+\cu_k)((\nabla v(\cu_n) - \nabla v(z+\cu_k))_{z+\cu_k}|^2 \nonumber \\
& \leq \avsum_{z\in 3^{k-1}\mathbb{Z}^d \cap \cu_n} |\s_*^{-1}(z+\cu_k)| \fint_{z+\cu_k} ((\nabla v(\cu_n) - \nabla v(z+\cu_k))\cdot \s ((\nabla v(\cu_n) - \nabla v(z+\cu_k)) \nonumber \\
& \leq \sup_{z\in 3^{k-1}\mathbb{Z}^d \cap \cu_n}|\s_*^{-1}(z+\cu_k)| \avsum_{z\in 3^{k-1}\mathbb{Z}^d \cap \cu_n} J(z+\cu_k,p,q) - J(\cu_n,p,q).
\end{align}
By our double-variable representation for~$J$ in~\eqref{e.J.bigA}
\begin{align*}
\avsum_{z\in 3^{k-1}\mathbb{Z}^d \cap \cu_n} J(z+\cu_k,p,q) - J(\cu_n,p, q) = \frac{1}{2}\begin{pmatrix}
-p \\ q
\end{pmatrix} \cdot \biggl( \avsum_{z\in 3^{k-1}\mathbb{Z}^d \cap \cu_n} \bfA(z+\cu_k) - \bfA(\cu_n) \biggr)\begin{pmatrix}
-p \\ q
\end{pmatrix}\,,
\end{align*}
so summing~\eqref{e.grad.av.step2} over the scales from~$k=n-l+1$ and applying H\"older's inequality in the sum, we have
\begin{align*}
\lefteqn{
\sum_{k=n-l+1}^n 3^{2s(k-n)} \sup_{z\in 3^{k-1}\mathbb{Z}^d \cap \cu_n}|\s_*^{-1}(z+\cu_k)| \avsum_{z\in 3^{k-1}\mathbb{Z}^d \cap \cu_m} J(z+\cu_k,p,q) - J(\cu_m,p,q)
} \qquad \qquad  \qquad & \\
& \leq \sum_{k=n-l+1}^n 3^{2(s-\epsilon)(k-n)} \sup_{z\in 3^{k-1}\mathbb{Z}^d \cap \cu_n}|\s_*^{-1}(z+\cu_k)| \\
& \qquad \times \sum_{k=n-l+1}^n 3^{2\epsilon(k-n)} \biggl| \avsum_{z\in 3^{k-1}\mathbb{Z}^d \cap \cu_n} \bfA(z+\cu_k) - \bfA(\cu_n) \biggr|(|p| + |q|)^2 \\
& \leq C \lambda_{s-\epsilon}^{-1}(\cu_n) \mathcal{G}_{\epsilon,l}(\cu_n)\,.
\end{align*}
For the second term in~\eqref{e.triangle} we use the averages in~\eqref{e.averages},
\begin{align*}
|(& \nabla v(z+\cu_k,p,q_n))_{z+\cu_k}| = |-p + \s_*^{-1}(z+\cu_k)(q_n +\k(z+\cu_k)p)| \\
& = |\s_*^{-1}(z+\cu_k)\k(z+\cu_k)p - \shom_*^{-1}\khom p + (\s_*^{-1}(z+\cu_k)-\shom_*^{-1})(\shom_* - \khom)p| \\
& \leq \bigl|\bfA(z+\cu_k) - \bfAhom\bigr||\shom_*-\khom|
\,,
\end{align*}
and thus 
\begin{align*}
\sum_{k=n-l+1}^n 3^{2s(k-n)}  \avsum_{z\in 3^{k}\mathbb{Z}^d \cap \cu_n} & |(\nabla v(z+\cu_k)_{z+\cu_k}|^2 \\
& \leq |\shom_*-\khom|^2\sum_{k=n-l+1}^{n} 3^{2s(k-n)} \avsum_{z\in 3^{k-1}\mathbb{Z}^d \cap \cu_n} \bigl|\bfA(z+\cu_k) - \bfAhom\bigr|^2 \\
& \leq C \mathcal{H}_{s,l}(\cu_n)\,.
\end{align*}
This completes the proof. 
\end{proof}

\begin{proposition}
\label{p.flux.norm}
For any~$\epsilon > 0$ there exists a constant~$C=C(d,\epsilon)$ such that for any~$t\in (0,1)$ and~$p\in\mathbb{R}^d$ with~$|p|=1$, if~$q$ is given by~\eqref{qs.choice} then
\begin{align*}
3^{-tn}\|\a\nabla v(\cu_n)\|_{\Hring^{-t}(\cu_n)} \leq C \Lambda_{t-\epsilon}^{\nicefrac12}(\cu_n)\|\s^{\nicefrac12}\nabla v(\cu_n)\|_{\L^2(\cu_n)}\,.
\end{align*}
\end{proposition}
\begin{proof}
The proof is nearly identical to step 1 in Proposition~\ref{p.grad.norm}. Using the coarse-graining inequality for the flux in~\eqref{e.cg.ineq}, we find
\begin{align*}
\sum_{k=-\infty}^{n} 3^{2t(k-n)} & \avsum_{z\in 3^{k-1}\mathbb{Z}^d\cap \cu_n} |(\a\nabla v(\cu_n))_{z+\cu_k}|^2 \\
& \leq \sum_{k=-\infty}^{n} 3^{2t(k-n)} \sup_{z\in 3^{k-1}\mathbb{Z}^d\cap \cu_n}|\b(z+\cu_k)| \avsum_{z\in 3^{k-1}\mathbb{Z}^d\cap \cu_n} |\b^{-\nicefrac12}(z+\cu_k)(\a\nabla v(\cu_n))_{z+\cu_k}|^2  \\
& \leq \|\s^{\nicefrac12}\nabla v(\cu_n)\|^2_{\L^2(\cu_n)} \sum_{k=-\infty}^{n} 3^{2t(k-n)}\sup_{z\in 3^{k-1}\mathbb{Z}^d\cap \cu_n} |\b(z+\cu_k)| \\
& \leq C\epsilon^{-1} \|\s^{\nicefrac12}\nabla v(\cu_n)\|^2_{\L^2(\cu_n)} \Lambda_{t-\epsilon}(\cu_n)\,.
\end{align*}
\end{proof}

Putting together our results we can now prove Proposition~\ref{p.ahom}.
\begin{proof}[Proof of Proposition~\ref{p.ahom}]
Let~$s,t$ be the parameters given by~\ref{a.ellipticity} and~$l\in\mathbb{N}$ fixed. We apply Proposition~\ref{p.J.bound}, Proposition~\ref{p.grad.norm} with parameter~$s+\epsilon$ and~$\epsilon = \frac{1-s-t}{2}$, and Proposition~\ref{p.flux.norm} with parameter~$t+\epsilon$ and~$\epsilon = \frac{1-s-t}{2}$. Then we obtain
\begin{align*}
& \limsup_{n\to\infty}J(\cu_n,p,q) \\
& \leq C(d,1-s-t)\limsup_{n\to\infty} \Bigl(\lambda_{s}^{-\nicefrac12}(\cu_n)\mathcal{G}_{\epsilon,l}^{\nicefrac12}(\cu_n) + \mathcal{H}_{s+\epsilon,l}^{\nicefrac12}(\cu_n) + 3^{-\epsilon l} \lambda_{s,l}^{-\nicefrac12}(\cu_n)\|\s^{\nicefrac12}\nabla v(\cu_n)\|_{\L^2(\cu_n)}\Bigr) \\
& \qquad \times \limsup_{n\to\infty} \bigl( \Lambda_{t}^{\nicefrac12}(\cu_n)\|\s^{\nicefrac12}\nabla v(\cu_n)\|_{\L^2(\cu_n)}\bigr)\,.
\end{align*}
By Proposition~\ref{p.Ahom} and~\eqref{e.J.bigA} we know that~$J(\cu_n,p,q)$ converges almost surely as~$n\to\infty$, so almost surely
\begin{align*}
\limsup_{n\to\infty}\|\s^{\nicefrac12}\nabla v(\cu_n)\|_{\L^2(\cu_n)} = \limsup_{n\to\infty}(2J(\cu_n,p,q,))^{\nicefrac12} < \infty\,.
\end{align*}
Similarly, by our ellipticity assumption~\ref{a.ellipticity} we have that almost surely
\begin{align*}
\limsup_{n\to\infty}\lambda_{s}^{-\nicefrac12}(\cu_n)<\infty \quad \mbox{and} \quad  \limsup_{n\to\infty} \Lambda_{t}^{\nicefrac12}(\cu_n) < \infty\,.
\end{align*}
For any~$l\in\mathbb{N}$ the random variable~$\mathcal{G}_{\epsilon,l}(\cu_n)$ is a sum of~$l$ random variables (one at each scale~$k$), each of which vanishes in the limit by Proposition~\ref{p.Ahom}. The exact same reasoning applies to~$\mathcal{H}_{s+\epsilon,l}(\cu_n)$, so for any~$l\in\mathbb{N}$ we have almost surely
\begin{equation*}
\lim_{n\to\infty} \mathcal{G}_{\epsilon,l}(\cu_n)  + \lim_{n\to\infty} \mathcal{H}_{s+\epsilon,l}(\cu_n)= 0\,.
\end{equation*}
Then sending first~$n\to\infty$ and then~$l\to\infty$ we conclude that
\begin{equation*}
\lim_{n\to\infty}J(\cu_n,p,(\shom_* -\khom)p) = 0\,.
\end{equation*}
Since the exact same arguments apply to the adjoint problem we conclude that~$J^*(\cu_n,p,(\shom_*+\khom)p)$ also vanishes in the limit, which in view of~\eqref{e.s.gap} completes the proof.
\end{proof}

\section{Homogenization of the Dirichlet problem}
\label{s.homogenize}

We define the homogenized matrix
\begin{equation}
\label{e.ahom.def}
\ahom \coloneqq \shom + \khom\,,
\end{equation}
which is almost surely the limit of both~$\s(\cu_n)+\k(\cu_n)$ and~$\s_*(\cu_n)+\k(\cu_n)$. As in the remark at the end of Section~\ref{ss.cg.defs}, we are allowed to make a centering operation by subtracting from the coefficient field~$\a(\cdot)$ any constant, skew-symmetric matrix~$\mathbf{h}$, because this leaves the set of solutions unchanged. The coarse-grained quantities~$\s(U)$ and~$\s_*(U)$ are then unchanged while~$\k(U)$ is replaced with~$\k(U) - \mathbf{h}$. Since we proved in Proposition~\ref{p.ahom} that the limiting matrix~$\khom$ is skew-symmetric we can take~$\mathbf{h} = \khom$ and therefore assume without loss of generality that~$\khom = 0$.

Convergence of the coarse-grained coefficients implies homogenization of the Dirichlet problem by deterministic arguments. We define the quantity that will appear in the bound by
\begin{equation}
\label{e.mathcalE.def}
\mathcal{E}_s(\cu_n) \coloneqq  \sum_{k=-\infty}^n 3^{2s(k-n)} \max_{z\in 3^k\mathbb{Z}^d \cap \cu_n} |\bfA(z+\cu_k) - \bfAhom| \,.
\end{equation}
It is immediate from the definitions~\eqref{e.lambda.def} and~\eqref{e.Lambda.def} that for any~$s\in (0,1)$ and~$n\in\mathbb{Z}$,
\begin{equation}
\label{e.lambda.by.mathcalE}
\lambda_s^{-1}(\cu_n) \leq C(|\shom^{-1}| + \mathcal{E}_s(\cu_n)) \quad \mbox{and} \quad \Lambda_s(\cu_n) \leq C(|\bhom| + \mathcal{E}_s(\cu_n))\,.
\end{equation}

\begin{theorem}
\label{t.black.box}
Suppose that~$s\in (0,1)$,~$\epsilon \in (0,s)$,~$\ahom$ is the homogenized matrix in~\eqref{e.ahom.def}, and~$\khom = 0$. There exists a constant~$C = C(\ahom,d,s,\epsilon)$ such that if~$h\in H^{1+s}(\cu_0)$ and~$u\in H^1_{\s}(\cu_0)$ satisfies
\begin{align}
\label{e.to.homogenize}
\left\{
\begin{aligned}
& \nabla \cdot \a \nabla u = \nabla \cdot \ahom\nabla h \quad &\mbox{in } \cu_0 \,, \\
& u=h \quad &\mbox{on } \partial \cu_0\,,
\end{aligned}
\right.
\end{align}
then 
\begin{align*}
\lefteqn{
\|\nabla u - \nabla h \|_{\underline{\hat{H}}^{-s}(\cu_0)} + \|\a\nabla u - \ahom \nabla h\|_{\underline{\hat{H}}^{-s}(\cu_0)}
}\qquad & \\ 
& \leq C\|\nabla h\|_{\underline{H}^s(\cu_0)}\sup_{z\in 3^n\mathbb{Z}^d \cap \cu_0}(1 + \mathcal{E}_{s-\epsilon}^{\nicefrac12}(z+\cu_n))\big(3^{sn} + \sup_{z\in 3^n\mathbb{Z}^d \cap \cu_0} \mathcal{E}_s^{\nicefrac{1}{2}}(z + \cu_n)\big)\,.
\end{align*}
\end{theorem}

Theorem~\ref{t.black.box} is a deterministic error bound using the coarse-grained ellipticity constants. Combining this with convergence of the coarse-grained matrices proves quenched homogenization for the Dirichlet problem. Theorem~\ref{t.main} is just a rescaling of the following corollary.

\begin{corollary}
\label{c.homogenization}
Assume~\ref{a.stationarity},\ref{a.ergodicity},~\ref{a.finite.E} and~\ref{a.ellipticity} and suppose that~$\khom=0$,~$\alpha \in (\max\{s,t\},1)$,~$h\in H^{1+\alpha}(\cu_0)$ and~$\a^\epsilon \coloneqq \a( \nicefrac{\cdot}{\epsilon})$. If for each~$\epsilon \in (0,1)$ we let~$u^\epsilon\in H^1_{\s}(\cu_0)$ denote the unique solution to
\begin{align}
\label{e.eqn.hom.cor}
\left\{
\begin{aligned}
& \nabla \cdot \a^\epsilon \nabla u^\epsilon = \nabla \cdot \ahom \nabla h  &\mbox{in } \cu_0\,, \\
& u^\epsilon = h &\mbox{on } \partial \cu_0\,,
\end{aligned}
\right.
\end{align}
then
\begin{equation*}
\mathbb{P}\biggl[ \limsup_{\epsilon \to 0} \bigl(\|\nabla u^\epsilon - \nabla h \|_{\underline{H}^{-\alpha}(\cu_0)} + \|\a^\epsilon\nabla u^\epsilon - \ahom \nabla h\|_{\underline{H}^{-\alpha}(\cu_0)}\bigr) = 0 \biggr] = 1\,.
\end{equation*}
\end{corollary}

Theorem~\ref{t.black.box} and Corollary~\ref{c.homogenization} are stated with~$\khom = 0$. If~$\khom$ does not vanish then, in view of the centring operation~$\a \to \a-\khom$, the flux quantity that we bound is
\begin{equation*}
\|(\a-\khom)\nabla u - \shom \nabla h\|_{\underline{H}^{-\alpha}(\cu_0)}\,,
\end{equation*}
and we recover the flux bound by
\begin{equation}
\|\a\nabla u - \ahom \nabla h\|_{\underline{H}^{-\alpha}(\cu_0)} \leq \|(\a-\khom)\nabla u - \shom \nabla h\|_{\underline{H}^{-\alpha}(\cu_0)} + |\khom|\|\nabla u - \nabla h\|_{\underline{H}^{-\alpha}(\cu_0)}\,.
\end{equation}

There are three steps in proving Theorem~\ref{t.black.box}. The first, in Proposition~\ref{p.homogenize}, is a duality argument which controls the difference between the heterogeneous and the homogenized solution by a quantity involving the difference of~$\a$ and~$\ahom$. The second step, Proposition~\ref{p.fluxdiff}, uses the coarse-graining inequalities to obtain an explicit dependence on~$\mathcal{E}_s$, which will tend to zero as a consequence of the ergodic theorem. The final step has already been accomplished in Proposition~\ref{p.boundary.data}, where we bound the energy of the solution by the data of the problem and the coarse-grained ellipticity constants.

\begin{proposition}
\label{p.homogenize}
Suppose that~$s\in (0,1)$,~$U\subseteq \cu_0$ is a bounded domain which is either~$C^{1,1}$ or convex and Lipschitz,~$\ahom$ is the homogenized matrix in~\eqref{e.ahom.def} with~$\khom = 0$, and~$h\in H^{1+s}(U)$. There exists a constant~$C=C(U,\ahom,d)$ such that if~$u\in H^1_{\s}(U)$ satisfies
\begin{align*}
\left\{
\begin{aligned}
& \nabla \cdot \a \nabla u = \nabla \cdot \ahom \nabla h &\mbox{in } U\,, \\
& u=h \quad &\mbox{on } \partial U\,,
\end{aligned}
\right.
\end{align*}
then
\begin{equation}
\label{e.duality.estimate}
\|\nabla u - \nabla h \|_{\Hminusuls{-s}(U)} + \|\a\nabla u - \ahom \nabla h\|_{\Hminusuls{-s}(U)} \leq C \|(\ahom - \a)\nabla u\|_{\Hminusuls{-s}(U)}\,.
\end{equation}
\end{proposition}
\begin{proof}
The proof is by duality, so suppose that~$\mathbf{g}\in C^\infty(U)$. Let~$w$ be the solution to
\begin{equation*}
\begin{cases}
-\nabla \cdot \ahom \nabla w = \nabla \cdot \mathbf{g} & \mbox{in } U\,, \\
w = 0 & \mbox{on } \partial U\,.
\end{cases}
\end{equation*}
Then
\begin{align*}
\fint_U \nabla (u-h) \cdot \mathbf{g} & = -\fint_U (u-h) \nabla \cdot \mathbf{g} = \fint_U (u-h) \nabla \cdot \ahom \nabla w = \fint_U \nabla \cdot \ahom\nabla (u-h) \, w \\
& = \fint_U \nabla \cdot (\ahom - \a)\nabla u\, w = \fint_U (\ahom - \a)\nabla u \cdot \nabla w \leq \|(\ahom - \a)\nabla u\|_{\Hminusuls{-s}(U)}\|\nabla w\|_{\underline{H}^s(U)}\,.
\end{align*}
We claim that our assumptions on~$U$ imply that~$\|\nabla w\|_{\underline{H}^s(U)} \leq C\|\mathbf{g}\|_{\underline{H}^s(U)}$ for any~$s\in (0,1)$. By~\cite[Theorems 2.4.2.5 and 3.1.2.1]{Grisvard} in a convex or~$C^{1,1}$ domain the solution to the Dirichlet problem~$\Delta w = \nabla \cdot \mathbf{g}$ with zero boundary data satisfies the estimate~$\|w\|_{H^2(U)} \leq C \|\mathbf{g}\|_{H^1(U)}$. By testing the equation for~$w$ with itself we also have~$\|w\|_{H^1(U)} \leq C\|\mathbf{g}\|_{L^2(U)}$, so the solution operator is a bounded linear operator from~$L^2(U) \to H^1(U)$ and~$H^1(U) \to H^2(U)$. Since the Sobolev spaces~$H^s(U)$ can be constructed as interpolation spaces (see~\cite[Chapter 7, p. 228]{Adams}) we may apply the interpolation theorem~\cite[Theorem 7.23]{Adams} to conclude that the solution operator is bounded from~$H^s(U) \to H^{1+s}(U)$.

This proves that
\begin{equation*}
\fint_U \nabla (u-h)\cdot \mathbf{g} \leq C\|(\ahom - \a)\nabla u\|_{\Hminusuls{-s}(U)} \|\mathbf{g}\|_{\underline{H}^s(U)}\,,
\end{equation*}
which completes the bound for the gradients by duality. For the fluxes we simply use the triangle inequality
\begin{equation*}
\|\a\nabla u - \ahom \nabla h\|_{\Hminusuls{-s}(U)} \leq \|\a\nabla u - \ahom \nabla u\|_{\Hminusuls{-s}(U)} + |\ahom|\|\nabla u - \nabla h\|_{\Hminusuls{-s}(U)}\,.
\end{equation*}
\end{proof}

\begin{proposition}
\label{p.fluxdiff}
Suppose that~$\nabla \cdot \a\nabla u = \nabla \cdot \mathbf{f}$ in~$\cu_0$. Then for any~$s\in (0,1)$ and~$\epsilon \in (0,s)$ there exists a constant~$C = C(\ahom,s,\epsilon,d)$ such that for any~$n\in\mathbb{Z}$ with~$n\leq 0$ we have
\begin{align*}
\| (\a-\ahom)\nabla u\|_{\underline{\hat{H}}^{-s}(\cu_0)}& \leq C \|\s^{\nicefrac12}\nabla u\|_{\L^2(\cu_0)} \sup_{z\in 3^n\mathbb{Z}^d \cap \cu_0}\mathcal{E}_s^{\nicefrac{1}{2}}(z+\cu_n) \\
& \quad + C 3^{sn}[\mathbf{f}]_{\underline{H}^{s}(\cu_0)}\sup_{z\in 3^n\mathbb{Z}^d \cap \cu_0}(1 + \mathcal{E}_{s-\epsilon}^{\nicefrac{1}{2}}(z+\cu_n))\,.
\end{align*}
\end{proposition}
\begin{proof}
Fix a scale~$n \leq 0$. For~$z\in 3^n\mathbb{Z}^d \cap \cu_0$, let~$w_z$ solve~$\nabla \cdot \a\nabla w_z = \nabla \cdot (\mathbf{f} - (\mathbf{f})_{z+\cu_n})$ in~$z+\cu_n$ with zero Dirichlet boundary condition. Then
\begin{align*}
\|(\a-\ahom)\nabla u \|_{\underline{\hat{H}}^{-s}(z+\cu_n)} & \leq \|(\a-\ahom)\nabla (u-w_z) \|_{\underline{\hat{H}}^{-s}(z+\cu_n)} \\ & \qquad + |\ahom|\|\nabla w_z\|_{\underline{\hat{H}}^{-s}(z+\cu_n)} + \|\a \nabla w_z\|_{\underline{\hat{H}}^{-s}(z+\cu_n)}\,.
\end{align*}
The first term we control by~\cite[Lemma A1]{AK.HC} and~\eqref{e.cg.diff}
\begin{align*}
\lefteqn{
\|(\a-\ahom)\nabla (u-w_z) \|_{\underline{\hat{H}}^{-s}(z+\cu_n)}^2
} \quad & \\
& \leq \sum_{k=-\infty}^n 3^{2sk} \avsum_{y \in 3^k\mathbb{Z}^d \cap (z+\cu_n)} | (\a-\ahom)\nabla (u-w_z) )_{y+\cu_k}|^2 \\
& \leq \sum_{k=-\infty}^n 3^{2sk} \avsum_{y \in 3^k\mathbb{Z}^d \cap (z+\cu_n) } \sup_{|p|=1} J(y+\cu_k,p,\ahom p)\|\s^{\nicefrac12}\nabla (u-w_z)\|_{\L^2(y+\cu_k)} \\
& \leq  3^{sn}\mathcal{E}_s(z+\cu_n) ( \|\s^{\nicefrac12}\nabla u\|_{\L^2(z+\cu_n)}^2 + \|\s^{\nicefrac12}\nabla w_z\|_{\L^2(z+\cu_n)}^2 )\,.
\end{align*}
Combining Proposition~\ref{p.gen.poincare} and Proposition~\ref{p.boundary.data},
\begin{align*}
\|\s^{\nicefrac12}\nabla w_z\|_{\L^2(z+\cu_n)} + \|\nabla w_z\|_{\underline{\hat{H}}^{-s}(z+\cu_n)} + \|\a \nabla w_z\|_{\underline{\hat{H}}^{-s}(z+\cu_n)} \leq C \frac{1 + \Lambda_{s}^{\nicefrac12}(z+\cu_n)}{\lambda_{s-\epsilon}^{\nicefrac12}(z+\cu_n)} 3^{sn} [\mathbf{f}]_{\underline{H}^{s}(z+\cu_n)}\,,
\end{align*}
and therefore, using also~\eqref{e.lambda.by.mathcalE},
\begin{align*}
\|(\a-\ahom)\nabla u \|_{\underline{\hat{H}}^{-s}(z+\cu_n)} & \leq 3^{sn}\mathcal{E}_s^{\nicefrac{1}{2}}(z+\cu_n) \|\s^{\nicefrac12}\nabla u\|_{\L^2(z+\cu_n)} \\
& \quad + C 3^{sn}(1 + \mathcal{E}_{s-\epsilon}^{\nicefrac{1}{2}}(z+\cu_n)) 3^{sn} [\mathbf{f}]_{\underline{H}^{s}(z+\cu_n)} \,.
\end{align*}
We partition~$\cu_0$ into cubes at scale~$3^n$ to conclude that
\begin{align*}
\| (\a-\ahom)\nabla u\|_{\underline{\hat{H}}^{-s}(\cu_0)} &  \leq 3^{-sn} \biggl(\avsum_{z\in 3^n\mathbb{Z}^d \cap \cu_0} \|(\a-\ahom)\nabla u\|_{\underline{\hat{H}}^{-s}(z+\cu_n)}^2 \biggr)^{\nicefrac12} \\
& \leq C \|\s^{\nicefrac12}\nabla u\|_{\L^2(\cu_0)} \sup_{z\in 3^n\mathbb{Z}^d \cap \cu_0}\mathcal{E}_s^{\nicefrac{1}{2}}(z+\cu_n) \\
& \quad + C 3^{sn}[\mathbf{f}]_{\underline{H}^{s}(\cu_0)}\sup_{z\in 3^n\mathbb{Z}^d \cap \cu_0}(1 + \mathcal{E}_{s-\epsilon}^{\nicefrac{1}{2}}(z+\cu_n)) \,.
\end{align*}
\end{proof}


It is straightforward to combine the last two propositions to obtain Theorem~\ref{t.black.box}, which we then use to prove Corollary~\ref{c.homogenization}.

\begin{proof}[Proof of Theorem~\ref{t.black.box}]
Suppose that
\begin{align*}
\left\{
\begin{aligned}
& \nabla \cdot \a \nabla u = \nabla \cdot \ahom\nabla h \quad &\mbox{in } \cu_0 \,, \\
& u=h \quad &\mbox{on } \partial \cu_0\,.
\end{aligned}
\right.
\end{align*}
Then combining Proposition~\ref{p.homogenize}, Proposition~\ref{p.fluxdiff}, Proposition~\ref{p.boundary.data} and~\eqref{e.lambda.by.mathcalE}
\begin{align*}
\lefteqn{
\|\nabla u - \nabla h \|_{\underline{\hat{H}}^{-s}(\cu_0)} + \|\a\nabla u - \ahom \nabla h\|_{\underline{\hat{H}}^{-s}(\cu_0)}
}\qquad \qquad \qquad & \\
& \leq C\|(\a-\ahom)\nabla u\|_{\underline{\hat{H}}^{-s}(\cu_0)} \\
& \leq C \|\s^{\nicefrac12}\nabla u\|_{\L^2(\cu_0)} \sup_{z\in 3^n\mathbb{Z}^d \cap \cu_0}\mathcal{E}_s^{\nicefrac{1}{2}}(z+\cu_n) \\
& \quad + C 3^{sn}[\ahom\nabla h]_{\underline{H}^{s}(\cu_0)}\sup_{z\in 3^n\mathbb{Z}^d \cap \cu_0}(1 + \mathcal{E}_{s-\epsilon}^{\nicefrac{1}{2}}(z+\cu_n)) \\
& \leq C(1 + \mathcal{E}_{s-\epsilon}(\cu_0) )\|\nabla h\|_{\underline{H}^s(\cu_0)}\sup_{z\in 3^n\mathbb{Z}^d \cap \cu_0}\mathcal{E}_s^{\nicefrac{1}{2}}(z+\cu_n) \\
& \quad + C 3^{sn}[\ahom\nabla h]_{\underline{H}^{s}(\cu_0)}\sup_{z\in 3^n\mathbb{Z}^d \cap \cu_0}(1 + \mathcal{E}_{s-\epsilon}^{\nicefrac{1}{2}}(z+\cu_n)) \\
& \leq C\|\nabla h\|_{\underline{H}^s(\cu_0)}\sup_{z\in 3^n\mathbb{Z}^d \cap \cu_0}(1 + \mathcal{E}_{s-\epsilon}^{\nicefrac12}(z+\cu_n))\big(3^{sn} + \sup_{z\in 3^n\mathbb{Z}^d \cap \cu_0} \mathcal{E}_s^{\nicefrac{1}{2}}(z + \cu_n))\big)\,,
\end{align*}
which concludes the proof.
\end{proof}

\begin{proof}[Proof of Corollary~\ref{c.homogenization}.]
Let~$\alpha \in (\max\{s,t\},1)$,~$h\in H^{1+\alpha}(\cu_0)$ and for each~$\epsilon \in (0,1)$, let~$u^\epsilon \in H^1_{\s}(\cu_0)$ denote the unique solution to~\eqref{e.eqn.hom.cor}. Given any domain~$V$, let~$\bfA^\epsilon(V)$ denote the coarse-graining of~$\a^\epsilon$ in~$V$; then directly from the definition of~$J(V,p,q)$ in~\eqref{e.J.def} it follows that~$\bfA^\epsilon(V) = \bfA(\epsilon^{-1}V)$. Theorem~\ref{t.black.box} yields, for~$\delta = \frac{\alpha - \max\{s,t\}}{2}$ and any~$n\in \mathbb{Z} \cap (-\infty, 0]$,
\begin{align}
\label{e.from.thm}
\lefteqn{
\|\nabla u - \nabla h \|_{\underline{\hat{H}}^{-\alpha}(\cu_0)} + \|\a\nabla u - \ahom \nabla h\|_{\underline{\hat{H}}^{-\alpha}(\cu_0)}
}\qquad & \notag \\ 
& \leq C\|\nabla h\|_{\underline{H}^{\alpha}(\cu_0)}\sup_{z\in 3^n\mathbb{Z}^d \cap \cu_0}(1 + \mathcal{E}_{\alpha-\delta}^{\nicefrac12}(\epsilon^{-1}z+\epsilon^{-1}\cu_n))\big(3^{\alpha n} + \sup_{z\in 3^n\mathbb{Z}^d \cap \cu_0} \mathcal{E}_{\alpha}^{\nicefrac{1}{2}}(\epsilon^{-1}z + \epsilon^{-1}\cu_n)\big)\,.
\end{align}

Let~$m = -\log_3 \epsilon$ and fix a scale separation~$l\in\mathbb{N}$. Then with the sum over~$\{k\in\mathbb{R}: 3^k = 3^{m+j}\,, j \in \mathbb{Z}\cap (-\infty,0]\}$ we have for any~$\alpha > \max\{s,t\}$
\begin{align}
\label{e.top.scales}
\mathcal{E}_\alpha(3^m z + \cu_{n+m}) &\leq 3^{-2\alpha l}\sum_{k=-\infty}^{n+m-l} 3^{2\alpha(k-n-m+l)} \sup_{y\in 3^k\mathbb{Z}^d \cap (3^m z + \cu_{n+m})} |\bfA(y+\cu_k) -\bfAhom| \notag \\
& \quad \quad + \sum_{k=n+m-l+1}^{n+m} 3^{2\alpha (k-m-n)} \sup_{y\in 3^k\mathbb{Z}^d \cap (3^m z + \cu_{n+m})} |\bfA(y+\cu_k) -\bfAhom| \notag \\
& \leq C3^{-2(\alpha - \max\{s,t\})l}\bigl(\Lambda_{t}(3^m z + \cu_{n+m}) +\lambda_{s}^{-1}(3^m z + \cu_{n+m})\bigr) \notag \\
& \quad \quad + \sum_{k=n+m-l+1}^n 3^{2\alpha(k-m-n)} \sup_{y\in 3^k\mathbb{Z}^d \cap (3^m z + \cu_{n+m})} |\bfA(y+\cu_k) -\bfAhom|\,.
\end{align}
In the second term there are finitely many random variables~$\bfA(y+\cu_k)$ at each of~$l$ scales. Moreover from the definition of~$y+\cu_k$, the dilations~$m\to\infty$ (corresponding to~$\epsilon \to 0$) simply dilate the domains~$y+\cu_k$. By Proposition~\ref{p.Ahom} each term converges to zero almost surely; it follows that the finite sum over~$l$ scales converges to zero almost surely as~$m\to\infty$ ($\epsilon \to 0$), and the same conclusion holds if we take a supremum of~$\mathcal{E}_\alpha(3^m z + \cu_{n+m})$ over the finite set~$z\in 3^n\mathbb{Z}^d \cap \cu_0$. To bound the first term in~\eqref{e.top.scales} we use
\begin{equation*}
\sup_{z\in 3^n\mathbb{Z}^d \cap \cu_0} \Lambda_t(3^m z + \cu_{n+m}) \leq 3^{-2tn}\Lambda_t(\cu_{m}) \quad \mbox{and} \quad \sup_{z\in 3^n\mathbb{Z}^d \cap \cu_0} \lambda_s^{-1}(3^m z + \cu_{n+m}) \leq 3^{-2sn}\lambda_s^{-1}(\cu_{m})\,,
\end{equation*}
which follows from the definitions in~\eqref{e.lambda.def} and~\eqref{e.Lambda.def}. Then  for any~$l\in\mathbb{N}$ we have
\begin{align*}
\lefteqn{
\limsup_{\epsilon \to 0}\sup_{z\in 3^n\mathbb{Z}^d \cap \cu_0} \mathcal{E}_{\alpha}(\epsilon^{-1}z + \epsilon^{-1}\cu_n)\big)
} & \\
&\qquad \qquad \leq C3^{-2(\alpha - \max\{s,t\})l} 3^{-2\max\{s,t\}n} \limsup_{m\to\infty} (\Lambda_t(\cu_m) + \lambda_s^{-1}(\cu_m) )\,.
\end{align*}
Choosing~$l(n) = -\gamma n$ for~$\gamma > \frac{\max\{s,t\}}{\alpha - \max\{s,t\}}$ we see that
\begin{equation}
\lim_{n\to -\infty}3^{-2(\alpha - \max\{s,t\})l(n)} 3^{-2\max\{s,t\}n} = 0\,,
\end{equation}
and from our ellipticity assumption
\begin{equation*}
\mathbb{P}\biggl[ \limsup_{m\to\infty} \bigl(\Lambda_{t}(\cu_m) +\lambda_{s}^{-1}(\cu_m)\bigr) < \infty \biggr] = 1\,.
\end{equation*}
Putting the above estimates together and sending first~$m\to\infty$ and then~$n\to -\infty$ in~\eqref{e.top.scales} we obtain
\begin{equation*}
\mathbb{P}\biggl[ \lim_{\epsilon \to 0}\sup_{z\in 3^n\mathbb{Z}^d \cap \cu_0} \mathcal{E}_{\alpha}(\epsilon^{-1}z + \epsilon^{-1}\cu_n)\big) = 0\biggr] = 1\,,
\end{equation*}
which concludes the proof of the corollary when combined with~\eqref{e.from.thm}.
\end{proof}

\appendix

\section{Functional inequalities}

Recall that the norm~$\|\cdot\|_{\Hring^{-s}(\cu_n)}$ is defined in~\eqref{e.Hring.def}.
\begin{proposition}
\label{p.poincare}
There exists a constant~$C=C(d)$ such that
\begin{equation*}
\|u-(u)_{\cu_n}\|_{\L^2(\cu_n)} \leq C \|\nabla u\|_{\underline{H}^{-1}(\cu_n)}\,,
\end{equation*}
and for any~$s\in (0,1]$ there exists a constant~$C=C(d,s)$ such that
\begin{equation*}
\|u-(u)_{\cu_n}\|_{\underline{H}^s(\cu_n)} \leq C \|\nabla u\|_{\Hring^{-(1-s)}(\cu_n)}\,.
\end{equation*}
\end{proposition}
\begin{proof}
The proof is a straightforward modification of~\cite[Lemma A.2, Lemma A.3]{AK.HC}.
\end{proof}

\begin{proposition}
\label{p.coarse.poincare}
For any~$s\in (0,1)$ there exists a constant~$C=C(d,s)$ such that for any~$n\in\mathbb{N}$ and~$u\in H^1_{\s}(\cu_n)$ satisfying~$-\nabla \cdot \a \nabla u = 0$ we have
\begin{align*}
\|\nabla u\|_{\Hring^{-s}(\cu_n)} & \leq C \lambda_{s}^{-\nicefrac12}(\cu_n) \|\s^{\nicefrac12}\nabla u\|_{\L^2(\cu_n)} \\
\|\a\nabla u\|_{\Hring^{-s}(\cu_n)} & \leq C \Lambda_{s}^{\nicefrac12}(\cu_n) \|\s^{\nicefrac12}\nabla u\|_{\L^2(\cu_n)}\,.
\end{align*}
\end{proposition}
\begin{proof}
See~\cite[Lemma 2.2]{AK.HC}.
\end{proof}

\begin{proposition}
\label{p.cutoff.func}
Suppose that~$n\in\mathbb{N}$ and~$\varphi \in C^\infty(\cu_n)$ such that~$\varphi|_{\cu_{n-1}} = 1$,~$\|\nabla \varphi\|_{L^\infty(\cu_n)} \leq C3^{-n}$ and $\|\nabla^2\varphi\|_{\L^\infty(\cu_n)} \leq C3^{-2n}$. Then for any~$t\in (0,1)$ and~$u$ such that~$(u)_{\cu_n} = 0$,
\begin{equation*}
\|u\nabla \varphi\|_{\underline{H}^t(\cu_n)} \leq C 3^{-n}\|\nabla u\|_{\underline{H}^{-(1-t)}(\cu_n)}\,.
\end{equation*}
\end{proposition}
\begin{proof}
By the triangle inequality we have in each sub-cube
\begin{align*}
\|u\nabla \varphi & - (u\nabla \varphi)_{z+\cu_k}\|_{\L^2(z+\cu_k)} \\
& \leq 2\|u -(u)_{z+\cu_k}\|_{\L^2(z+\cu_k)} \|\nabla \varphi\|_{L^\infty(\cu_n)} + |(u)_{z+\cu_k}|\|\nabla \varphi -(\nabla \varphi)_{z+\cu_k}\|_{\L^2(z+\cu_k)}\,,
\end{align*}
so
\begin{align*}
\sum_{k=-\infty}^n 3^{-2tk} & \avsum_{z\in 3^{k-1}\mathbb{Z}^d, z+\cu_k \subseteq \cu_n} \|u\nabla \varphi - (u\nabla \varphi)_{z+\cu_k}\|_{\L^2(z+\cu_k)}^2 \\
& \leq C \|\nabla \varphi\|_{L^\infty(\cu_n)}^2 \sum_{k=-\infty}^n 3^{-2tk} \avsum_{z\in 3^{k-1}\mathbb{Z}^d, z+\cu_k \subseteq \cu_n} \|u- (u)_{z+\cu_k}\|_{\L^2(z+\cu_k)}^2 \\
& \qquad + C\sum_{k=-\infty}^n 3^{-2tk} \avsum_{z\in 3^{k-1}\mathbb{Z}^d, z+\cu_k \subseteq \cu_n} \|u\|_{\L^2(z+\cu_k)}^2 3^{2k}\|\nabla^2 \varphi\|_{L^\infty(\cu_n)}^2\,.
\end{align*}
Now by Proposition~\ref{p.poincare} with~$s=1$ and~\eqref{e.Hring.def} we control the first term by
\begin{align*}
\|\nabla \varphi\|_{L^\infty(\cu_n)}^2 \sum_{k=-\infty}^n & 3^{-2tk} \avsum_{z\in 3^{k-1}\mathbb{Z}^d, z+\cu_k \subseteq \cu_n} \|u- (u)_{z+\cu_k}\|_{\L^2(z+\cu_k)}^2 \\
& \leq 3^{-2n} \sum_{k=-\infty}^n 3^{-2tk} \avsum_{z\in 3^{k-1}\mathbb{Z}^d, z+\cu_k \subseteq \cu_n} \sum_{j=-\infty}^k 3^{2j} \avsum_{z'\in 3^{j-1}\mathbb{Z}^d, z'+\cu_j \subseteq z+\cu_k} |(\nabla u)_{z'+\cu_j}|^2 \\
& \leq 3^{-2n} \sum_{j=-\infty}^n 3^{2j} \sum_{k=j}^n 3^{-2tk} \avsum_{z'\in 3^{j-1}\mathbb{Z}^d, z'+\cu_j \subseteq \cu_n} |(\nabla u)_{z'+\cu_j}|^2 \\
& \leq C(t) 3^{-2n} \sum_{j=-\infty}^n 3^{2(1-t)j}\avsum_{z'\in 3^{j-1}\mathbb{Z}^d, z'+\cu_j \subseteq \cu_n} |(\nabla u)_{z'+\cu_j}|^2 \,.
\end{align*}
For our other term we use Proposition~\ref{p.poincare} with~$s=1-t$ to bound
\begin{align*}
\sum_{k=-\infty}^n 3^{-2tk} \avsum_{z\in 3^{k-1}\mathbb{Z}^d, z+\cu_k \subseteq \cu_n} \|u\|_{\L^2(z+\cu_k)}^2 3^{2k}\|\nabla^2 \varphi\|_{L^\infty(\cu_n)}^2 & \leq C(t)3^{-4n} \|u\|_{\L^2(\cu_n)}^2 \\
& \leq C(t)3^{-2n}\|\nabla u\|^2_{\underline{H}^{-(1-t)}(\cu_n)}\,.
\end{align*}
\end{proof}

\begin{proposition}
\label{p.multi.poincare}
For every~$s\in (0,1)$ there exists a constant~$C = C(d,s)$ such that for every~$n\in\mathbb{Z}$
\begin{equation}
\label{e.multi.poincare}
\|f\|_{\underline{\hat{H}}^{-s}(\cu_n)} \leq C\biggl(\sum_{k=-\infty}^{n} 3^{2s(k-n)} \avsum_{y\in 3^k\mathbb{Z}^d \cap \cu_n} |(f)_{y+\cu_{k}}|^2 \biggr)^{\nicefrac12}\,.
\end{equation}
\begin{proof}
This is~\cite[Lemma A.1]{AK.HC} in the case~$p=q=2$.
\end{proof}
\end{proposition}

\section{Proof of Example~\ref{ex.field}}
\label{appendix.ex}

In this appendix we prove the claim~\eqref{e.field.in.Besov} from Example~\ref{ex.field}. Recall that~$\{\varphi_j(\cdot)\}_{j=1}^\infty$ is a sequence of independent Gaussian fields such that on each cube~$z+\cu_j$ the field~$\varphi_j(\cdot)$ is a constant, given by a centred Gaussian random variable with variance~$\sigma^2$, and the values of~$\varphi_j$ on different cubes are independent. We let~$W_j(\cdot) = e^{\varphi_j(\cdot) - \nicefrac{\sigma^2}{2}}$, and define
\begin{equation*}
f_m(\cdot) =  \prod_{j=1}^m W_j(\cdot)
\qquad
\mbox{and}
\qquad
f(\cdot) = \sum_{m=1}^\infty \frac{f_m(\cdot)}{m^{3}}\,,
\end{equation*}
where the convergence is in~$L^1(\Omega; L^1_{\mathrm{loc}}(\mathbb{R}^d))$. Our goal is to prove the following:
\begin{proposition}
\label{p.example.Besov}
If~$\sigma < 1$ and
\begin{equation}
\label{e.t.condition}
\frac{\sigma}{4}(\sqrt{8d}-\sigma) < t < 1
\end{equation}
then
\begin{equation}
\mathbb{P}\biggl[ \limsup_{n\to\infty} \|f\|_{\Bring_{\infty,1}^{-2t}(\cu_n)} < \infty \biggr] = 1\,.
\end{equation}
\end{proposition}
We will bound the norm by splitting it into three parts. For each~$m,n\in \mathbb{N}$ such that~$m \leq n$ we define
\begin{equation}
\label{e.Am.def}
A_{m}(\cu_n) := \sum_{k=-\infty}^{m-1} 3^{2t(k-n)} \max_{z\in 3^k\mathbb{Z}^d \cap \cu_n } \biggl|\fint_{z+\cu_k} f_m \biggr|\,,
\end{equation}
and
\begin{equation}
\label{e.Bm.def}
B_{m}(\cu_n) := \sum_{k=m}^n 3^{2t(k-n)} \max_{z\in 3^k\mathbb{Z}^d \cap \cu_n} \biggl|\fint_{z+\cu_k} f_m \biggr|\,,
\end{equation}
and we bound
\begin{equation}
\label{e.triangle.norm}
\|f\|_{\Bring_{\infty,1}^{-2t}(\cu_n)} \leq \sum_{m=1}^{\lceil \nicefrac{n}{2}\rceil} \frac{A_m(\cu_n)}{m^3} + \sum_{m=1}^{\lceil \nicefrac{n}{2} \rceil} \frac{B_m(\cu_n)}{m^3} + \sum_{m=\lceil \nicefrac{n}{2} \rceil+1}^\infty \frac{\|f_m\|_{\infty,1}^{-2t}(\cu_n)}{m^3}\,.
\end{equation}
In Lemma~\ref{l.A.terms} we will prove that the~$A_m(\cu_n)$ terms are small because of the scale discount factor. In Lemma~\ref{l.B.terms} we will prove that the~$B_m(\cu_n)$ terms are bounded because the scale separation~$m \leq \lceil \nicefrac{n}{2} \rceil$ ensures that on large scales the field~$f_m$ should concentrate around its average. Finally, we will prove in Lemma~\ref{l.expectation.uniform} that the Besov-type norms of~$f_m$ are uniformly controlled in expectation, and therefore the tail of the sum will be small. We will use the Gaussian integral calculation
\begin{equation}
\label{e.Gaussian.int}
\mathbb{E}[W_j^p] = e^{\nicefrac{p(p-1)\sigma^2}{2}}\,,
\end{equation}
and Rosenthal's inequality~\cite[Theorem 3]{Rosenthal1970}, which states that for every~$p\geq 2$ there exists a constant~$C=C(p)$, such that if~$X_1,\cdots, X_N$ are independent random variables then
\begin{equation*}
\mathbb{E}\biggl[ \biggl|\sum_{i=1}^N X_i \biggr|^p \biggr]^{\nicefrac{1}{p}} \leq C \max\bigg\{ \biggl( \sum_{i=1}^N \mathbb{E}[ |X_i|^p ] \biggr)^{\nicefrac{1}{p}},  \biggl( \sum_{i=1}^N \mathbb{E}[ |X_i|^2 ] \biggr)^{\nicefrac{1}{2}} \bigg\}\,.
\end{equation*}
Adding the averaging factors in the sums and controlling the second term by the~$L^p$ norm we obtain
\begin{equation}
\label{e.Rosenthal}
\mathbb{E}\biggl[ \biggl|\avsum_{i=1}^N X_i \biggr|^p \biggr]^{\nicefrac{1}{p}} \leq \frac{C}{N^{\nicefrac12}}\biggl( \avsum_{i=1}^N \mathbb{E}[ |X_i|^p ] \biggr)^{\nicefrac{1}{p}}\,,
\end{equation}
which is the form we will reference. We begin by proving a lemma which uses elements of the proof of~\cite[Theorem 6.1]{JSW2019}.
\begin{lemma}
\label{l.pth.moment}
If~$p \geq 2$ and
\begin{equation}
\label{e.sigma.cond.one}
\sigma^2 < \frac{d}{p-1}\,,
\end{equation}
then there exists a constant~$C=C(p,\sigma^2,d)$ such that for every~$k\in \mathbb{N}$
\begin{equation}
\mathbb{E}\biggl[\biggl| \fint_{\cu_k} f_{k-1} \biggr|^p \biggr]^{\nicefrac{1}{p}} \leq C\,.
\end{equation}
\end{lemma}
\begin{proof}
We first decompose our field as a sum and use the triangle inequality in the expectation to get
\begin{align}
\label{e.triangle.it}
\mathbb{E}\biggl[\biggl| \fint_{\cu_k} W_0\cdots W_{k-1} \biggr|^p\biggr]^{\nicefrac{1}{p}}
& \leq
\mathbb{E} \biggl[\biggl|\sum_{j=0}^{k-1} \fint_{\cu_k} W_{j+1}\cdots W_{k-1}(W_j - 1) \biggr|^p\biggr]^{\nicefrac{1}{p}} + 1
\notag \\ & \leq
\sum_{j=0}^{k-1} \mathbb{E}\biggl[  \biggl| \avsum_{z\in 3^j\mathbb{Z}^d \cap \cu_k} \fint_{z+\cu_j} W_{j+1}\cdots W_{k-1}(W_j - 1) \biggr|^p \biggr]^{\nicefrac{1}{p}} + 1\,,
\end{align}
where for~$j=k-1$ we interpret the integrand as simply~$W_{k-1} - 1$.

For each~$j \in \{0,\ldots,k-1\}$ we let~$\mathbb{E}_j$ denote the expectation conditioned on~$W_{j+1},\ldots, W_{k-1}$ (which for~$j=k-1$ is just the unconditioned expectation). Since~$p\geq 2$ and the fields are independent on different cubes, we can apply the Rosenthal inequality~\eqref{e.Rosenthal} to obtain
\begin{align*}
\lefteqn{
\mathbb{E}_j \biggl[ \biggl|\avsum_{z\in 3^j\mathbb{Z}^d \cap \cu_k} \fint_{z+\cu_j} W_{j+1}\cdots W_k(W_j - 1) \biggr|^p \biggr]
} \quad \quad & \\
& \leq
C(p) 3^{-\frac{dp}{2}(j-k)} \avsum_{z \in 3^j\mathbb{Z}^d \cap \cu_k} \mathbb{E}_j \biggl[\biggl|\fint_{z+\cu_j} W_{j+1}\cdots W_k(W_j - 1) \biggr|^p\biggr]\,.
\end{align*}
Since the fields~$W_l$ with~$l\geq j$ are constant on each box~$z+\cu_j$, we can factor them out of the integral. Then taking an expectation of the above expression, using independence of the~$W_j$ and~\eqref{e.Gaussian.int},
\begin{align*}
\mathbb{E}\biggl[  \avsum_{z\in 3^j\mathbb{Z}^d \cap \cu_k} \biggl|\fint_{z+\cu_j} W_{j+1}\cdots W_k(W_j - 1) \biggr|^p \biggr]
& \leq
C(p) 3^{-\frac{dp(k-j)}{2}}\mathbb{E}[(W_{j+1}\cdots W_k)^p]  \mathbb{E}\bigl[\bigl| W_j - 1\bigr|^p\bigr]
\\ & \leq
C(p,\sigma^2) 3^{-(k-j)\bigl( \frac{dp}{2}-\frac{\sigma^2 p (p-1)}{2}\bigr)}\,.
\end{align*}
The condition~\eqref{e.sigma.cond.one} ensures that this is summable from~$j=0$ to~$k-1$. Then substituting in~\eqref{e.triangle.it} and summing concludes the proof.
\end{proof}

Note that by controlling the supremum over sub-cubes by the~$\ell^p$ norm and correcting for the averaging factor we have, for any~$p\in [1,\infty)$,
\begin{equation}
\label{e.B.for.appendix}
A_m(\cu_n) \leq \sum_{k=-\infty}^{m-1} 3^{(2t-\nicefrac{d}{p})(k-n)} \biggl(\avsum_{z\in 3^k\mathbb{Z}^d \cap \cu_n} \biggl| \fint_{z+\cu_k} f_m \biggr|^p \biggr)^{\nicefrac{1}{p}} \,,
\end{equation}
and similarly for~$B_m(\cu_n)$. We will use this expression for an optimized choice of~$p$ in each of the three following lemmas.

\begin{lemma}
\label{l.A.terms}
If~$\sigma \leq 1$,
\begin{equation}
\frac{\sigma}{4}(\sqrt{8d}-\sigma) < t <1\,,
\end{equation}
and~$m\leq n$ then
\begin{equation}
\mathbb{P}\bigl( A_m(\cu_n) \geq \delta \bigr) \leq \frac{C}{\delta} 3^{(2t-\sigma\sqrt{\nicefrac{d}{2}})(m-n)}
\end{equation}
\end{lemma}
\begin{proof}
Fix~$p \geq 2$, to be specified later such that it satisfies~\eqref{e.sigma.cond.one}. By H\"older's inequality, independence of the fields,~\eqref{e.Gaussian.int} and Lemma~\ref{l.pth.moment},
\begin{align}
\label{e.E.of.A}
\lefteqn{
\mathbb{E}\biggl[ \sum_{k=-\infty}^{m-1} 3^{(2t-\nicefrac{d}{p})(k-n)} \biggl( \avsum_{z\in 3^k\mathbb{Z}^d \cap \cu_n} \biggl|\fint_{z+\cu_k} f_m\biggr|^p \biggr)^{\nicefrac{1}{p}}\biggr]
} \qquad & \notag \\
& =
\mathbb{E}\biggl[ \sum_{k=-\infty}^{m-1} 3^{(2t-\nicefrac{d}{p})(k-n)} \biggl( \avsum_{z\in 3^k\mathbb{Z}^d \cap \cu_n} (W_k \cdots W_m)^p \biggl| \fint_{z+\cu_k} f_{k-1} \biggr|^p \biggr)^{\nicefrac{1}{p}}\biggr]
\notag \\ & \leq
\sum_{k=-\infty}^{m-1} 3^{(2t-\nicefrac{d}{p})(k-n)} \biggl( \avsum_{z\in 3^k\mathbb{Z}^d \cap \cu_n} \mathbb{E}[(W_k \cdots W_m)^p] \mathbb{E}\biggl[\biggl| \fint_{z+\cu_k} f_{k-1} \biggr|^p\biggr] \biggr)^{\nicefrac{1}{p}}
\notag \\ & \leq
C\sum_{k=-\infty}^{m-1} 3^{(2t-\nicefrac{d}{p})(k-n)} 3^{\frac{\sigma^2(p-1)}{2}(m-k)} \mathbb{E}\biggl[\biggl| \fint_{\cu_k} f_{k-1} \biggr|^p \biggr]^{\nicefrac{1}{p}}
\notag \\ & \leq
C 3^{(2t-\nicefrac{d}{p})(m-n)} \sum_{k=-\infty}^{m-1} 3^{(2t-\nicefrac{d}{p} - \frac{\sigma^2(p-1)}{2})(k-m)}\,.
\end{align}
We now optimize in~$p$, for a given~$\sigma^2$, by choosing~$p = \sqrt{\frac{2d}{\sigma^2}}$, which we note satisfies~\eqref{e.sigma.cond.one} because
\begin{equation*}
\sigma^2 \leq \sigma \leq \sqrt{\frac{d}{2}}\sigma \leq \frac{d}{p-1}\,.
\end{equation*}
We conclude by applying the Markov inequality to~$A_{m}(\cu_n)$ and bounding the expectation by~\eqref{e.E.of.A}.
\end{proof}

\begin{lemma}
\label{l.B.terms}
If~$\sigma^2 \leq \nicefrac{t}{4}$ then there exist constants~$C = C(\sigma^2,t,d)$ and~$D = D(\sigma^2,t,d)$ such that
\begin{equation}
\mathbb{P}\biggl( B_m(\cu_n) \geq D \biggr) \leq C 3^{t (m-n)}
\end{equation}
\end{lemma}
\begin{proof}
First we fix~$p = \frac{2d}{t}$. We note that~$\sigma^2 < \frac{d}{2p-1}$ so that we can apply Lemma~\ref{l.pth.moment} with parameter~$2p$. We will use the bound~\eqref{e.B.for.appendix} in the analogous form for~$B_m(\cu_n)$, so it is convenient to define the random variable
\begin{equation*}
X_{m,k}(\cu_n) = \avsum_{z\in 3^k\mathbb{Z}^d \cap \cu_n} \biggl( \biggl| \fint_{z+\cu_k} (f_m-1) \biggr|^p - \mathbb{E}\biggl[\biggl| \fint_{\cu_k} (f_m-1) \biggr|^p\biggr]\biggr)\,,
\end{equation*}
because using Lemma~\ref{l.pth.moment} and the triangle inequality
\begin{align}
\label{e.split.B}
B_m(\cu_n) \leq \sum_{k=m}^n 3^{(2t-\nicefrac{d}{p})(k-n)} X_{m,k}^{\nicefrac{1}{p}}(\cu_n) + C(p,\sigma^2,d,s)\,.
\end{align}
In order to control the first term above we use the Markov inequality, Rosenthal's inequality in the form~\eqref{e.Rosenthal}, noting that there are~$3^{d(n-k)}$ elements~$z\in 3^k \mathbb{Z}^d \cap \cu_n$, and finally Lemma~\ref{l.pth.moment}, to bound
\begin{align*}
\mathbb{P}\biggl[ X_{m,k}(\cu_n) \geq \delta \biggr]
& \leq \frac{1}{\delta^2 3^{d(n-k)}}\mathbb{E}\biggl[\biggl( \biggl| \fint_{\cu_k} (f_m-1) \biggr|^p - \mathbb{E}\biggl[\biggl| \fint_{\cu_k} (f_m-1) \biggr|^p\biggr]\biggr)^2\biggr]
\\ & \leq
\frac{1}{\delta^2 3^{d(n-k)}}\mathbb{E}\biggl[ \biggl| \fint_{\cu_k} (f_m-1) \biggr|^{2p} \biggr]
\\ & =
\frac{1}{\delta^2 3^{d(n-k)}}\mathbb{E}\biggl[ \biggl| \avsum_{z\in 3^m\mathbb{Z}^d \cap \cu_k} \fint_{z+\cu_m} (f_m-1) \biggr|^{2p} \biggr]
\\ & \leq
\frac{C}{\delta^2 3^{d(n-k)} 3^{\frac{dp(k-m)}{2}}} \mathbb{E}\biggl[ \biggl| \fint_{\cu_m} (f_m-1) \biggr|^{2p} \biggr]
\\ & \leq
\frac{C}{\delta^2 3^{d(n-k)} 3^{\frac{dp}{2}(k-m)}}\,.
\end{align*}
Then
\begin{align*}
\mathbb{P}\bigl( \exists k \in [m,n] : X_{m,k}(\cu_n) \geq 3^{(tp-d)(n-k)} \bigr)
& \leq
\sum_{k=m}^n \mathbb{P}\bigl(X_{m,k}(\cu_n) \geq 3^{(tp-d)(n-k)} \bigr)
\\ & \leq
C\sum_{k=m}^n 3^{(2tp-d)(k-n)}3^{\frac{dp}{2}(m-k)}
\end{align*}
In the last factor we can bound~$3^{\frac{dp}{2}(m-k)} \leq 3^{t p (m-k)}$, so that if~$D = \sum_{k=-\infty}^n 3^{(t-\nicefrac{d}{p})(k-n)}$ then
\begin{align*}
\mathbb{P}\biggl( \sum_{k=m}^n 3^{(2t-\nicefrac{d}{p})(k-n)} X_{m,k}^{\nicefrac{1}{p}}(\cu_n) \geq D  \biggr)
& \leq
\mathbb{P}\bigl( \exists k \in [m,n] : X_m(\cu_n) \geq 3^{(tp-d)(n-k)} \bigr)
\\ & \leq
C 3^{t(m-n)}\,.
\end{align*}
Combining this with~\eqref{e.split.B} concludes the proof.
\end{proof}

\begin{lemma}
\label{l.expectation.uniform}
If~$\sigma \leq 1$ and
\begin{equation}
\label{e.sigma.B3}
\frac{\sigma}{4}(\sqrt{8d}-\sigma) < t < 1
\end{equation}
then there exists a constant~$C = C(\sigma^2,t,d)$ such that for every~$m,n \in \mathbb{N}$
\begin{equation*}
\mathbb{E}[ \| f_m\|_{\Bring_{\infty,1}^{-2t}(\cu_n)} ] \leq C\,.
\end{equation*}
\end{lemma}
\begin{proof}
First we note that if~$m\geq n$ then by factoring out the~$W_j$ with~$j\geq n$ we obtain
\begin{equation}
\mathbb{E}[\|f_m\|_{\Bring_{\infty,1}^{-2t}(\cu_n)}] = \mathbb{E}[\|f_{n-1}\|_{\Bring_{\infty,1}^{-2t}(\cu_n)}]\,.
\end{equation}
At scale~$k$ we factor out the fields~$W_j$ with~$j\geq k$ which we then estimate by~\eqref{e.Gaussian.int}. Then exactly as in~\eqref{e.E.of.A} but summing over the full range of scales, we obtain
\begin{align}
\label{e.the.sum}
\mathbb{E}\biggl[ \sum_{k=-\infty}^n 3^{(2t-\nicefrac{d}{p})(k-n)} \biggl( \avsum_{z\in 3^k\mathbb{Z}^d \cap \cu_n} \biggl|\fint_{z+\cu_k} f_m\biggr|^p \biggr)^{\nicefrac{1}{p}}\biggr]
\leq
C \sum_{k=-\infty}^n 3^{(2t-\nicefrac{d}{p} - \frac{\sigma^2(p-1)}{2})(k-n)}\,.
\end{align}
Choosing~$p = \sqrt{\nicefrac{2d}{\sigma^2}}$, the condition that the sum in~\eqref{e.the.sum} be finite is our assumption~\eqref{e.sigma.B3}.
\end{proof}

\begin{proof}[Proof of Proposition~\ref{p.example.Besov}]
We decompose~$\|f\|_{\Bring_{\infty,1}^{-2t}(\cu_n)}$ as in~\eqref{e.triangle.norm}. Using Lemma~\ref{l.A.terms}
\begin{align*}
\mathbb{P}\biggl( \sum_{m=1}^{\lceil \nicefrac{n}{2} \rceil} \frac{A_m(\cu_n)}{m^3} \geq C\delta \biggr) & \leq \sum_{m=1}^{\lceil \nicefrac{n}{2} \rceil} \mathbb{P}(A_m(\cu_n) \geq \delta m^3) \leq \sum_{m=1}^{\lceil n^\beta \rceil} \frac{C}{\delta m^3} 3^{(2t-\sigma\sqrt{\nicefrac{d}{2}})(m-n)} \\
& \leq \frac{C}{\delta} 3^{-\frac{n}{2}(2t-\sigma\sqrt{\nicefrac{d}{2}})}\,.
\end{align*}
This is summable in~$n$, so by the Borel Cantelli lemma and since~$\delta > 0$ was arbitrary, it follows that
\begin{equation}
\mathbb{P}\biggl( \limsup_{n\to\infty}  \sum_{m=1}^{\lceil \nicefrac{n}{2} \rceil} \frac{A_m(\cu_n)}{m^3} = 0 \biggr) = 1\,.
\end{equation}
We conclude similarly that for~$D$ as in Lemma~\ref{l.B.terms},
\begin{equation}
\mathbb{P}\biggl( \limsup_{n\to\infty}  \sum_{m=1}^{\lceil \nicefrac{n}{2} \rceil} \frac{B_m(\cu_n)}{m^3} \leq D \biggr) = 1\,.
\end{equation}
Finally, since
\begin{equation}
\sum_{m=\lceil \nicefrac{n}{2} \rceil + 1}^\infty m^{-3}\mathbb{E}[\|f_m\|_{\Bring_{\infty,1}^{-2t}(\cu_n)}] \leq C \sum_{m=\lceil \nicefrac{n}{2} \rceil + 1}^\infty \frac{1}{m^3}\leq \frac{C}{n^{2}}
\end{equation}
is summable in~$n$, we obtain
\begin{equation}
\mathbb{P}\biggl( \limsup_{n\to\infty}  \sum_{m=\lceil \nicefrac{n}{2} \rceil + 1}^\infty m^{-3}\|f_m\|_{\Bring_{\infty,1}^{-2t}(\cu_n)} = 0 \biggr) = 1\,,
\end{equation}
which concludes the proof.
\end{proof}

\subsubsection*{\bf Acknowledgments}
The author was partially supported by NSF grant DMS-2350340. The author acknowledges the support of the Natural Sciences and Engineering Research Council of Canada (NSERC) through a PGS-D award (598693-2025). L'auteur remercie le Conseil de recherches en sciences naturelles et en g\'enie du Canada (CRSNG) de son soutien (PGS-D 598693-2025).

\printbibliography

\end{document}